\documentclass{article}
 \usepackage[letterpaper]{geometry}
\usepackage[utf8]{inputenc}
\usepackage[T1]{fontenc}
\usepackage{url}
\usepackage[cmex10]{amsmath}

\providecommand{\email}[1]{\href{mailto:#1}{\nolinkurl{#1}\xspace}}

\def\withcolors{0}
\def\withnotes{0}
\usepackage[usenames,dvipsnames,table]{xcolor}
\usepackage{varwidth}
\usepackage{hyperref}%
\usepackage{url} %
\usepackage{booktabs} %
 \usepackage{%
     amssymb,%
     amsthm,%
     bbm,%
     etoolbox,%
     mathtools,%
     stmaryrd%
 }
\usepackage{amsfonts} %
\usepackage{nicefrac}%
\usepackage{mathtools}
\usepackage[T1]{fontenc}
\usepackage{lmodern}

\usepackage{tangocolors}
\usepackage{framed}
\usepackage{xcolor}
\usepackage{colortbl}

\usepackage{bbm}

\usepackage{mleftright}
\usepackage[mathscr]{eucal}

\newtheorem{lem}{Lemma}
\newtheorem{thm}{Theorem}
\newtheorem{defn}{Definition}
\newtheorem{rem}{Remark}
\newtheorem{assum}{Assumption}
\newtheorem{clm}{Claim}

\usepackage{amsfonts}
\usepackage{amssymb}
\usepackage{enumitem}
\newlist{steps}{enumerate}{1}
\setlist[steps, 1]{label = \quad \quad \quad Step \arabic*:}
\usepackage{tikz,pgf,pgfplots}
\usepgflibrary{shapes}
\usetikzlibrary{%
    arrows,%
    decorations.text,%
    positioning,%
    scopes,%
    shapes,
    patterns%
}

\usepackage[ruled]{algorithm2e} 
\newcommand{\ep}{\mathcal{E}}
\newcommand{\eqdef}{:=}
\newcommand{\eq}[1]{\begin{align*}#1\end{align*}}
\newcommand{\norm}[1]{\|#1\|}%

\newcommand{\R}{\mathbb{R}}
\newcommand{\N}{\mathbb{N}}
\newcommand{\E}[1]{\mathbb{E}\left[#1\right]}
\newcommand{\EE}[2]{\mathbb{E}_{#1}\left[#2\right]}
\newcommand{\EEC}[2]{\EE{}{ #1 \;\middle\vert\; #2 }}

\newcommand{\bPr}[1]{\Pr\mleft(#1\mright)}

 \newcommand{\B}{\mathscr{B}}
 
 \newcommand{\G}{\mathcal{G}} 
\newcommand{\V}{\mathcal{V}}
\newcommand{\X}{\mathcal{X}}
\newcommand{\Y}{\mathcal{Y}}
\newcommand{\W}{\mathcal{W}}
\newcommand{\oO}{\mathcal{O}}

\newcommand{\cZ}{\mathcal{Z}}
\newcommand{\uRoman}[1]{\uppercase\expandafter{\romannumeral#1}}

\newcommand{\Gconvex}{\mathcal{G_{\tt c}}}
\newcommand{\Gstrongconvex}{\mathcal{G_{\tt sc}}}
\newcommand{\p}{\mathbf{p}}
\newcommand{\cX}{\mathcal{X}}
\newcommand{\cY}{\mathcal{Y}}

\newcommand{\priv}{\varepsilon}

\ifnum\withcolors=1
  \newcommand{\new}[1]{{\color{red} {#1}}} %
  \newcommand{\newer}[1]{{\color{orange!50!red} {#1}}} %
  \newcommand{\newest}[1]{{\color{green!60!black} {#1}}} %
  \newcommand{\newerest}[1]{{\color{ta3plum} {#1}}}
  \newcommand{\mcolor}[1]{{\color{green}#1}} %
  \newcommand{\tcolor}[1]{{\color{blue}#1}} %
  \newcommand{\ccolor}[1]{{\color{purple}#1}} %
  \newcommand{\acolor}[1]{{\color{cyan}#1}} %
\else
  \newcommand{\new}[1]{{{#1}}}
  \newcommand{\newer}[1]{{{#1}}}
  \newcommand{\newest}[1]{{{#1}}}
  \newcommand{\newerest}[1]{{{#1}}}
  \newcommand{\mcolor}[1]{{#1}}
  
  \newcommand{\tcolor}[1]{{#1}}
  \newcommand{\ccolor}[1]{{#1}}
  \newcommand{\acolor}[1]{{#1}}
\fi

\ifnum\withnotes=1
  \usepackage{todonotes}
  \newcommand{\mnote}[1]{\par\mcolor{\textbf{M: }\sf #1}} %
  \newcommand{\tnote}[1]{\par\tcolor{\textbf{T: }\sf #1}} %
  \newcommand{\cnote}[1]{\par\ccolor{\textbf{C: }\sf #1}} %
  \newcommand{\anote}[1]{\par\acolor{\textbf{J: }\sf #1}} %
  \newcommand{\marginnote}[1]{\todo[size=\tiny,color=white,linecolor=black]{{#1}}}
  \newcommand{\amargin}[1]{\marginnote{\acolor{#1}}} %
  \newcommand{\cmargin}[1]{\marginnote{\ccolor{#1}}} %
  \newcommand{\mmargin}[1]{\marginnote{\mcolor{#1}}} %
  \newcommand{\tmargin}[1]{\marginnote{\tcolor{#1}}} %
\else
  \newcommand{\mnote}[1]{}
  \newcommand{\tnote}[1]{}
  \newcommand{\cnote}[1]{}
  \newcommand{\anote}[1]{}
  \newcommand{\marginnote}[1]{\ignore{#1}}
  \newcommand{\amargin}[1]{}
  \newcommand{\cmargin}[1]{}
  \newcommand{\mmargin}[1]{}
  \newcommand{\tmargin}[1]{}
  
\fi
\newcommand{\ignore}[1]{\leavevmode\unskip} %

\newcommand{\indic}[1]{\mathbbm{1}_{#1}}
\DeclareMathOperator{\sign}{sign}

\newcommand{\dtv}{\operatorname{d}_{\rm TV}}
\newcommand{\totalvardistrestr}[3][]{{\dtv^{#1}\mleft({#2, #3}\mright)}}
\newcommand{\totalvardist}[2]{\totalvardistrestr[]{#1}{#2}}
\newcommand{\kldiv}[2]{{\operatorname{D}\mleft({#1 \| #2}\mright)}}
\newcommand{\mutualinfo}[2]{ I\mleft(#1 \land #2\mright) }
\newcommand{\condmutualinfo}[3]{ I\mleft(#1 \land #2\mid #3\mright) }

\usepackage[numbers]{natbib}

\definecolor{red1}{rgb}{0.4,0,0}
\hypersetup{
colorlinks,
 citecolor=blue,
  linkcolor=red1,
  urlcolor=Blue
}
\title{Information-constrained optimization: can adaptive processing of gradients help?}
\author{%
Jayadev Acharya\thanks{Supported by NSF-CCF-1846300 (CAREER), NSF-CCF-1815893.} \\
Cornell University\\
\email{acharya@cornell.edu}\\
\and Cl\'ement L. Canonne\thanks{Part of this work was done while at IBM Research, supported by a Goldstine Postdoctoral Fellowship.}\\
University of Sydney\\
\email{clement.canonne@sydney.edu.au}
\and Prathamesh Mayekar\thanks{Supported by Wipro Ph.D. fellowship.}\\
Indian Institute of Science\\
\email{prathamesh@iisc.ac.in}\\
\and    Himanshu Tyagi\thanks{Supported by a grant from Robert Bosch Center for Cyber Physical Systems (RBCCPS), Indian Institute of Science.}\\ 
Indian Institute of Science\\
\email{htyagi@iisc.ac.in}\\
}
\date{}

\begin{document}

\maketitle

\begin{abstract}
  We revisit first-order optimization under local information constraints such as local privacy, gradient quantization, and computational constraints limiting access to a few coordinates of the gradient. In this setting, the optimization algorithm is not allowed to
  directly access the complete output of the gradient oracle, but only gets limited information about it subject to the local information constraints. 
  
  We study the role of adaptivity in processing the gradient output to obtain this limited information from it.
  We consider optimization for both convex and strongly convex functions and obtain tight or nearly tight lower bounds for the convergence rate, when adaptive gradient processing is allowed.
Prior work was restricted to convex functions and allowed only nonadaptive processing of gradients.
  For both of these function classes and for the three information constraints mentioned above, our lower bound implies that adaptive processing of gradients cannot outperform nonadaptive processing in most regimes of interest.
  We complement these results by exhibiting a natural optimization problem under information constraints for which adaptive processing of gradient strictly outperforms nonadaptive processing.

\end{abstract}

\newpage
  \tableofcontents
\newpage

\section{Introduction}
  Distributed optimization has emerged as a central tool in federated learning for building statistical and machine learning models for \newerest{distributed data}.
In addition, large scale optimization is typically implemented in a distributed fashion    
over multiple machines or multiple cores within the same
machine. These distributed implementations fit naturally in the oracle
framework of first-order optimization
(see~\cite{nemirovsky1983problem}) where in each iteration a user or
machine computes the gradient oracle output.  Due to practical local
constraints such as communication bandwidth, privacy concerns, or
computational issues, the entire gradient cannot be made available to
the optimization algorithm. 
Instead, the gradients must be passed
through a mechanism which, respectively, ensures privacy of user data
(local privacy constraints); or compresses them to a small number of
bits (communication constraints); or only computes a few coordinates
of the gradient (computational constraints).
Motivated by these
applications, we consider first-order optimization under such local
information constraints placed on the gradient oracle.

When designing a first-order optimization algorithm under local
  information constraints, one not only needs to design the
  optimization algorithm itself, but also the algorithm for local
  processing of the gradient estimates. \new{Many such algorithms have
  been proposed in recent years; see, for
  instance, \cite{duchi2014privacy}, \cite{abadi2016deep}, \cite{agarwal2018cpsgd}, \cite{gandikota2019vqsgd}, \cite{subramani2020enabling}, \cite{girgis2020shuffled},
  and the references therein for privacy constraints; \cite{seide20141}, \cite{alistarh2017qsgd}, \cite{suresh2017distributed}, \cite{konevcny2018randomized}, \cite{faghri2020adaptive}, \cite{ramezani2019nuqsgd}, \cite{lin2020achieving}, \cite{acharya2019distributed}, \cite{chen2020breaking}, \cite{huang2019optimal}, \cite{mayekar2020ratq}, \cite{mayekar2020limits}, \cite{safaryan2020uncertainty},
  and the references therein for communication
  constraints; \cite{nesterov2013introductory,richtarik2012parallel}
  for computational constraints.}  However, these algorithms primarily
  consider \emph{nonadaptive} procedures for gradient processing (with
  the exception of \cite{faghri2020adaptive}): that is, the scheme
  used to process the gradients at any iteration cannot depend
  on \newerest{the} information gleaned from previous
  iterations. \new{As a result, the following question remains largely
  open:}
\begin{center}
{\em Can adaptively processing gradients improve convergence in
information-constrained optimization?}
\end{center}

In this paper, we study this question for optimization over both
convex and strongly convex function families and under the three
different local constraints mentioned above: local privacy,
communication, and computational. For each of these constraints, we
establish lower bounds on convergence rates which hold even when the
gradients are adaptively processed. \new{In the next few sections,
we \newerest{discuss} prior related work and elaborate on our results
and techniques.}

  \subsection{Prior work}
    The \newerest{framework we consider} can be viewed as \newerest{an}
extension of the classical query complexity model
in~\cite{nemirovsky1983problem}.  Without information
constraints,~\cite{agarwal2012information} \newerest{provide a general
recipe for proving convex optimization lower bounds for different
function families in this model}.  Specifically, they
reduce \newerest{optimization problems} with a
first-order \newerest{oracle} to a mean estimation problem whose
probability of error is lower bounded using Fano's method
($cf.$~\cite{yu1997assouad}).
While our work, too, relies on a reduction to mean estimation, we deviate from the
prior
approach, using instead Assouad's method to prove lower bounds for various function families. This different approach in turn enables us to derive lower bounds for adaptive processing of gradients.

In the information-constrained setting, motivated by privacy concerns,~\cite{duchi2014privacy} consider \newest{the} problem where the gradient estimates must pass through a locally differentially private (LDP) channel. However, in their setting the LDP channels for \newerest{all} time steps are \newest{selected} at the start of optimization algorithm~--~in other words, the channel selection strategy is nonadaptive. Similarly,~\cite{mayekar2020ratq} and~\cite{mayekar2020limits} consider a similar problem and impose the constraint that the gradient estimates be quantized to a fixed number of bits. They, too, fix the quantization channels used at each time step at the start of optimization algorithm. In contrast, in this paper, we allow for {\em adaptive} channel selection strategies; as a result, the lower bounds established in these papers do not apply to our setting, and are more restrictive than our bounds.

The results of Duchi and Rogers~\cite{DuchiR19} for Bernoulli product distributions could be combined with our construction to obtain tight lower bounds for optimization in $p \in [1,2]$ under LDP constraints, but would not extend to the entire range of $p$. The work of Braverman, Garg, Ma, Nguyen, and Woodruff~\cite{BGMNW:16} on communication constraints, also for $p \in [1,2]$, is relevant as well; however, their bounds on mutual information cannot be applied directly, as their setting (Gaussian distributions) would not satisfy our almost sure gradient oracle assumption.

\newest{\cite{faghri2020adaptive} provide adaptive quantization schemes for convex and $\ell_2$ Lipschitz function family. While the worst-case convergence guarantees for the quantizers in \cite{faghri2020adaptive} are similar to those in \cite{alistarh2017qsgd} and \cite{mayekar2020ratq}, it shows some practical improvements over the state-of-the-art for some specific problem instances. This suggests that while adaptive quantization may not help in the worst case for non-smooth convex optimization, it may be useful for a smaller subclass of convex optimization problems.}

  \subsection{Our contributions}
    \newerest{We model the information constraints using a family of channels $\W$; see Section~\ref{ssec:info-constraints} for a description of the channel families corresponding to our constraints of interest. We consider first-order optimization where the output of the gradient oracle must be passed through a channel $W$ selected from $\W$.}
Specifically, the gradient is sent as input to this channel $W$, and the algorithm receives the output \newerest{of the channel}.
In each iteration of the algorithm, the channel to be used in that iteration can be selected
adaptively based on previously received channel outputs by the algorithm; or channels to
be used throughout can be fixed upfront, nonadaptively. 
The detailed problem setup is given in Section~\ref{ssec:set-up}. We obtain general lower bounds for optimization of convex and strongly convex functions using $\W$,
when adaptivity is allowed. 
These bounds are then applied to the specific constraints of interest to obtain our main results. 

\newer{Our first contribution is in showing that adaptive gradient processing does {not} help for some of the most typical optimization problems. Namely, we} prove that \newest{for most regimes of} local privacy, communication, or computational constraints, adaptive gradient processing has \newest{nearly} the same convergence rate as nonadaptive gradient processing for both convex and strongly convex function families. 
\new{As a consequence, this shows that the nondaptive  LDP algorithms from  \cite{duchi2014privacy} and nonadaptive compression protocols from \cite{mayekar2020ratq},~\cite{mayekar2020limits} are optimal for private and communication-constrained optimization, respectively, \newerest{even if adaptive gradient processing is allowed}. In another direction, under computational constraints, where we are allowed to compute only one gradient coordinate, we show that standard Random Coordinate Descent ($cf.$~\cite[Section 6.4]{bubeck2015convex}), which employs uniform (nonadaptive) sampling of gradient coordinates, is optimal for both the convex and strongly convex function families. This proves that adaptive sampling of gradient coordinates does not improve over nonadaptive sampling strategies.}

As previously discussed, prior work in both the locally private and communication-constrained settings concerned itself with the family of convex functions, with no lower bounds known for the more restricted family of \emph{strongly convex} functions, even for nonadaptive gradient processing protocols. The key obstacle is the fact that during the reduction from optimization to mean estimation, the known hard instance
for the strongly convex family, even when analyzed for nonadaptive protocols, leads to an estimation problem using adaptive protocols; and thus the lack of known lower bounds for adaptive information-constrained estimation prevented this approach from succeeding.
In more detail, this hard instance has gradients that can depend on the query point which in turn can be chosen
based on previously observed channel outputs, an issue which does not arise in the case of the convex family where the lower bounds are derived using affine functions for which the gradients do not depend on the query point. \newer{We manage to circumvent this issue by relying on a different reduction, which lets us  capitalize on a recent lower bound for adaptive mean estimation. Crucially, this recent lower bound does apply to adaptive estimation algorithms as well.
This lets us derive lower bounds for both convex Lipschitz \emph{and} strongly convex functions under adaptive gradient processing.}

\newerest{These lower bounds are seen to match the performance of existing algorithms in most settings, even in settings which were not considered in prior works. 
For optimization of convex Lipchitz functions over an $\ell_1$ ball using $r$ bits per
gradient query, prior work was restricted to the case $r=O(d)$ only. We show that a simple uniform quantizer used along with repeated queries of the same point is rate-optimal.}

\newerest{The results discussed above show that adaptive processing of gradients does not help
for convex optimization over $\ell_p$ balls or even strongly convex optimization over $\ell_2$ balls.} This raises the question of whether there are natural function families where adaptive gradient processing can lead to significant savings. Our third contribution is to provide an example of such a family. Specifically, we 
  exhibit a natural optimization problem (entailing $\ell_2$ minimization) under computational constraints
for which adaptive gradient processing provides a polynomial factor improvement in convergence rates compared to nonadaptive processing. 
The key feature of this optimization problem is that
\newer{the resulting gradients have \emph{structured sparsity}; adaptivity then allows for a two-phase optimization procedure, where the algorithm first ``explores'' to find the structure before, in a second phase, ``exploiting'' it to obtain more focused information about the function to minimize. However, nonadaptive gradient processing protocols cannot exploit this hidden structure, as finding it is now akin to locating a needle in a haystack; and thus exhibit much slower convergence rates.}

  \subsection{Organization.} The rest of the paper is organized as follows. After formally introducing in Section~\ref{sec:preliminaries} the setting, the function classes considered (convex and strongly convex), and the information constraints we are concerned with, we state and discuss our lower bounds in  Section~\ref{sec:results}.

In more detail, Section~\ref{sec:results:ldp} focuses on locally differentially private (LDP) optimization, and contains our theorems for convex functions (Theorems~\ref{t:conpriv_P2} and~\ref{t:conpriv_inf} for \newer{$p\in[1,2)$} and $p\in[2,\infty]$, respectively), as well as our lower bound for strongly convex functions (Theorem~\ref{t:sconpriv_P2}).
Section~\ref{sec:results:communication} contains the analogous results for optimization under communication constraints (Theorems~\ref{t:concom_P2} and~\ref{t:concom_inf} for convex functions, and Theorem~\ref{t:sconcom_P2} for strongly convex functions). 
Section~\ref{sec:results:rcd} focuses on optimization with $\ell_2$ loss under computational constraints (\textit{i.e.}, RCD-type schemes), with the lower bound of Theorem~\ref{t:concompt_P2} for convex functions and that of Theorem~\ref{t:sconcompt_P2} for strongly convex functions. Proofs of these lower bounds are given in Section~\ref{sec:proofs}.

Finally, Section~\ref{sec:proofs:adapt:advantage} discusses our example for which adaptive gradient processing {does} help, with Theorem~\ref{t:sparse:nonadapt} stating the lower bound for nonadaptive schemes and Theorem~\ref{t:sparse:adapt} providing an upper bound (significantly smaller) for adaptive ones.

  \paragraph{Notation.} Throughout the paper, $q$  denotes the H\"{o}lder
conjugate of $p$ (that is, $\frac{1}{p}+\frac{1}{q}=1$). We write  $a \vee
b$ and $a \wedge b$ for $\max\{a,b\}$ and $\min\{a, b\}$, respectively. We
use $\log$ for the binary
logarithm  and $\ln$ for the natural logarithm. Information-theoretic quantities, such as mutual information and Kullback--Leibler (KL) divergence, are defined using $\ln$. The iterated logarithm $\ln^*(a)$ is defined as the number of times $\ln$ must be iteratively
 applied to $a$ before the result is at most $1$.  Finally, we write $\{e_1, \dots, e_d\}$ for 
  the standard basis of $\R^d$.

\section{Setup and preliminaries}
\label{sec:preliminaries}
\subsection{Optimization under information constraints}
\label{ssec:set-up}
We consider the
problem of minimizing an unknown convex function $f\colon\X
\to \R$ over its domain $\X$ 
using \emph{oracle access}
to noisy subgradients of the
function.
That is, the algorithm is not directly given access to the function
but can get subgradients of the function at different points of its choice.
This class of optimization algorithms includes various descent algorithms,
which \newerest{often provide} optimal convergence rate among
all the algorithms in this class ($cf.$~\cite{nemirovsky1983problem}).

In our setup, gradient estimates supplied by the oracle must pass through a channel $W$,\footnote{\new{A channel $W$ with input alphabet $\X$ and output alphabet $\Y$,
  denoted $W\colon \X \to \Y$, represents
  the conditional distribution of the output of
  a randomized function given its input. In particular, $W(\cdot \mid x)$ is
  the conditional distribution of the channel given that the input is $x\in \cX$.}
}
\new{chosen by the algorithm
from a fixed set of channels $\W$,} and the optimization algorithm $\pi$ only has access to the output of this channel.
\new{The \emph{channel family} $\W$ represents information constraints
imposed in our distributed setting.}
In detail, the framework is as follows:
\begin{enumerate}
\item At iteration $t$, the first-order optimization algorithm $\pi$ makes a query for point $x_t$ to the oracle $O$.
\item Upon receiving the point $x_t$, the oracle outputs $\hat{g}(x_t)$, where $\EEC{\hat{g}(x_t)}{ x_t} \in  \partial f (x_t)$ and $\partial f (x_t)$ is the subgradient set of function $f$ at $x_t$.
\item \newerest{The subgradient estimate $\hat{g}(x_t)$ is passed through a channel \new{$W_t\in \W$} and the output $Y_t$ is observed by the first-order optimization algorithm. The algorithm then uses all the messages $\{Y_i\}_{i \in [t]}$ to further update $x_t$ to $x_{t+1}$.}
\end{enumerate}

\newerest{Let $\Pi_T$ be the set of all first-order optimization algorithms that are allowed $T$ queries to the oracle $O$ and after the $t$th query gets back  the output $Y_t$ with distribution $W_{t}(\cdot \mid \hat{g}(x_t))$.} 

Our goal is to select \newerest{gradient processing} channels $W_t$s and an optimization algorithm $\pi$ to \newerest{guarantee a} small worst-case optimization error. Two classes of \emph{channel selection strategies} are of interest: \emph{adaptive} and \emph{nonadaptive}. 

\medskip
\noindent\textit{Adaptive gradient processing.}
\newerest{Under adaptive gradient processing, the channel $W_t$ selected at time $t$ may depend on the previous outputs of channels $\{W_i\}_{i \in [t-1]}$.
  Specifically, denoting by $Y_t$ the output of the channel used at time $t$,
  which takes values in the output alphabet $\Y_t$,
the \emph{adaptive channel selection strategy} $S\eqdef (S_1, \ldots, S_T)$ over $T$ iterations 
consists of mappings $S_t\colon \Y^{t-1} \to \W$ that take $Y_1, \ldots Y_{t-1}$ as input and output a channel $W_t\in\W$ as output.
We write $\mathcal{S}_{\W, T}$ for the collection of all such channel selection strategies.
}

\medskip
\noindent\textit{Nonadaptive gradient processing.}
\newerest{Under nonadaptive selection, all the channels
 $\{W_{t}\}_{t \in [T]}$ through which the gradient estimates must pass are decided at the start of the optimization algorithm. {In other words, the $W_t$s are independent of the $t-1$ gradient observations received by the optimization algorithm until step $t$.} Denote the class of all nonadaptive strategies by $\mathcal{S}_{\W, T}^{\rm{}NA}$.}

We measure the performance of
an optimization protocol $\pi$ and a channel selection strategy $S$ for a given function $f$ and oracle $O$ using the metric
$\ep(f, O, \pi, S)$ defined as  
\begin{equation}
    \label{eq:def:metric}
  \ep(f, O, \pi, S) = \E{f(x_T)-\min_{x\in \X} f(x)},
\end{equation} %
where the expectation is over the randomness in $x_T$.

For various function and oracle classes, denoted by $\oO$, the channel constraint family $\W$, and the number of \newerest{iterations} $T$, we will characterize the \emph{adaptive minmax optimization error}
\begin{equation}
    \label{eq:def:optimization:error}
  \ep^\ast(\X , \oO, T, \W) =   \inf_{\pi \in \Pi_T}\inf_{S \in \mathcal{S}_{\W, T}} \sup_{(f, O)\in\oO}\ep(f, O, \pi, S)\,,
\end{equation}
and the corresponding \emph{nonadaptive minmax optimization error} 
\begin{equation}
    \label{eq:def:NAoptimization:error}
  \ep^{\rm{}NA \ast}(\X , \oO, T, \W) =   \inf_{\pi \in \Pi_T}\inf_{S \in \mathcal{S}_{\W, T}^{\rm{}NA}} \sup_{(f, O)\in\oO}\ep(f, O, \pi, S)\, .
\end{equation}
\new{Since the adaptive channel selection strategies include the 
  nonadaptive ones, we have $\ep^{\rm{}NA \ast}(\X , \oO, T, \W) \geq \ep^\ast(\X , \oO, T, \W).$ }

\subsection{Function classes}
\newerest{We now define the function classes and the corresponding oracles that we consider.}

\paragraph{Convex and $\ell_p$ Lipschitz function family.}
 Our first set of function families are parameterized by \new{a
   number} $p \in [1, \infty]$. Throughout, we restrict
 ourselves to convex functions over \new{a domain} $\X$, i.e.,
functions $f$ satisfying
 \begin{align} \label{e:convexity}
 f(\lambda x +(1-\lambda)y)\leq \lambda f(x) +(1-\lambda)f(y), \quad
 \forall x, y \in \X, \quad \forall \lambda \in [0, 1].
\end{align} 
 Further, for a family parameterized by $p$, we assume that the
 subgradient estimates returned by the first-order oracle for a
 function $f$ satisfy the following two assumptions:
 \begin{align}
\label{e:asmp_unbiasedness}
\E{\hat{g}(x)\mid x} \in \partial f(x), \quad \text{(Unbiased estimates)} \\
\label{e:asmp_as_bound}
\bPr{\norm{\hat{g}(x)}_q^2 \leq B^2\mid x}=1, \quad \text{(Bounded estimates)}
\end{align}
where  $\partial f(x)$ is the set of subgradient for $f$
at $x$ and $q\eqdef p/(p-1)$ is, \newer{as mentioned earlier,} the H\"older conjugate of $p$. We
denote by $\oO_{{\tt c}, p}$ the set of all pairs of functions and oracles satisfying Assumptions \eqref{e:convexity}, \eqref{e:asmp_unbiasedness}, and \eqref{e:asmp_as_bound}.

We note that \eqref{e:asmp_unbiasedness} is standard in stochastic optimization literature ($cf.$ \cite{nemirovsky1983problem}, \cite{nemirovski1995information}, \cite{bubeck2015convex}, \cite{agarwal2012information}).
To prove convergence guarantees on first-order optimization in the
classic setup (without any information constraints on the oracle), it
is enough to assume $\E{\norm{\hat{g}(x)}_q^2} \leq B^2$. We make a
slightly stronger assumption in this case \new{since the more relaxed
  assumption leads to technical difficulties in finding unbiased
  quantizers for gradients; see~\cite{mayekar2020ratq,mayekar2020limits}.
}

Note that by \eqref{e:asmp_unbiasedness} and \eqref{e:asmp_as_bound}
for every $x\in \cX$
there exists a vector $g\in \partial f(x)$ such that $\norm{g}_q\leq B$.
Further, since $f$ is convex,
$f(x)-f(y)\leq g^T(x-y)$ for every $g\in \partial f(x)$, whereby
$|f(x)-f(y)|\leq B\norm{x-y}_p$.
Namely, $f$ is $B$-Lipschitz continuous in the $\ell_p$ norm.\footnote{The
  same could be said under the weaker assumption
  $\E{\norm{\hat{g}(x)}_q^2 }\leq B^2$.}  

\newerest{
\begin{rem}[Convergence rate
for convex functions]\label{r:c_p1}
Without any information constraints (when gradient estimates are directly observed), upper bounds of $\frac{c_1DB \sqrt{\log d}}{\sqrt{T}} $ and $\frac{c_1DBd^{1/2-1/p}}{\sqrt{T}} $ \newerest{on the error} are achievable for $\ell_1$ and $\ell_p$, $p \in [2, \infty]$, convex family, respectively. Moreover, these rates are orderwise optimal. In particular, from \cite[Appendix C]{agarwal2012information} we have the following result: For $p=1$, stochastic mirror descent algorithm with mirror map $\Phi_a(x)= \frac{1}{a-1} \norm{x}_a^2$, where $a= \frac{2\log d}{2\log d-1},$ achieves the orderwise convergence rate; for $p \in [2, \infty],$ stochastic gradient descent achieves the orderwise optimal convergence rate.
\end{rem}}

\paragraph{Strongly convex and $\ell_2$ Lipschitz function family.}
We now consider a special subset of the  convex and $\ell_2$ Lipschitz
family described above, where the functions are strongly convex. Recall that for
$\alpha>0$, a function $f$ is \emph{$\alpha$-strongly convex on $\X$} if the following function $h$ is convex:
\begin{align}\label{e:strong_convexity}
h(x)=f(x) - \frac{\alpha}{2}\norm{x}_2^2, \quad \forall x \in \X.
\end{align}
\newerest{We denote by  $\oO_{{\tt sc}}$ the set of all pairs of functions and oracles satisfying \eqref{e:convexity}, \eqref{e:asmp_unbiasedness}, \eqref{e:strong_convexity}, and \eqref{e:asmp_as_bound} for $q=2$.}

The strong convexity parameter $\alpha$ is related to the parameter
$B$, the upper bound on the
$\ell_2$ norm of the gradient estimate. We state a relation between
them when the domain $\X$ contains an $\ell_{\infty}$
ball of radius $D$ centered at the origin; this property will be used when we derive lower bounds.  
\begin{lem}
	\label{lemma:B_alpha}
For any $ \X \supseteq \{x: \norm{x}_{\infty} \leq D\}$, we have
$\frac{B}{\alpha}\geq \frac{D d^{1/2}}{4}.$
\end{lem}

\begin{rem}[Convergence rate
for strongly convex functions]\label{r:sc} 
Without information constraints, stochastic gradient descent \newer{achieves} an upper bound of $  \frac{2B^2}{T+1} $ (\cite{nemirovski1995information}) for the strongly convex family, and this rate is optimal; see~\cite{agarwal2012information}.
\end{rem}  
\subsection{Information constraints}\label{ssec:info-constraints}
We describe  three specific constraints of interest to us: local
privacy, communication, and computation. The first two are
well-studied;
the third \newerest{is new} and arises in procedures such as random coordinate descent.
\paragraph{Local differential privacy.}

\new{To model local privacy,}
we define  the
$\priv$-locally differentially private (LDP) channel family $\W_{{\tt priv},
  \priv}$.
\begin{defn}\label{d:priv_constraints}
A channel $W\colon \R^d \to \R^d$
 is \emph{$\priv$-locally differentially private} ($\priv$-LDP) if  for all $x,x^{\prime}\in\R^d$,
\[	\frac{W (Y \in S \mid X =x)}{ W(Y \in S \mid X =x^{\prime})} \leq e^{\priv}\]
for all Borel measurable subsets $S$ of $\R^d$.
We denote by $\W_{{\tt priv}, \priv}$ the set of all $\priv$-LDP channels.
\end{defn}
\new{When operating under local privacy constraints,} 
 the oracle's subgradient estimates are passed through an $\priv$-LDP channel, and only the output is available to
 the optimization algorithm. \newer{Thus, the resulting process which handles the data of individual users, accessed in each
oracle query, is overall differentially private, a notion of privacy extensively studied and widely used in practice.}   
 \paragraph{Communication constraints.}
\new{To model communication constraints,}
 we define the $\W_{{\tt com}, r}$, the $r$-bit communication-constrained channel family, as follows.
\begin{defn}\label{d:comm_constraints}
A channel $W\colon\R^d \to \{0,1\}^r$ constitutes an \emph{$r$-bit communication-constrained} channel.
We denote by $\W_{{\tt com}, r}$ the set of all $r$-bit communication-constrained channels.
\end{defn}

\paragraph{Computational constraints.}
For high-dimensional optimization, altogether computing the
subgradient estimates \newest{can be computationally expensive}. Often in
such cases, one resorts to computing only a few coordinates of the
gradient estimates and using only them for optimization
(\cite{nesterov2013introductory,richtarik2012parallel}). This
motivates the oblivious sampling channel family $\W_{\tt obl}$, where
the optimization algorithm gets to see only one
\new{randomly chosen}
coordinate of the gradient estimate. 
\begin{defn} 
\label{def:oblv}
An \emph{oblivious sampling} channel $W$ is a channel
$W\colon\R^d\to\R^d$
specified by a probability vector $(p_i)_{i\in[d]}$, $i.e.$,
a vector $p$ 
such that
$p_i\geq 0$ for all $i$ and $\sum_{i \in [d]} p_i=1$. 
For an input $g \in \R^d$, the output distribution of $W$ is given by
$W( g(i) e_i \mid g) = p_i, \forall i \in [d]$,
where $e_1,\dots, e_d$ denote the standard basis vectors. 
We denote by  $\W_{\tt obl}$ the set of all oblivious sampling channels.
\end{defn}

Therefore, at most one coordinate of
the oracle's the gradient estimate can be used by the optimization
algorithm. Further, this coordinate is sampled obliviously to the
input gradient estimate itself. \new{We note that the special case of
  $p_i=\frac{1}{d}$ $\forall\, i\in [d]$
  corresponds to sampling employed by standard \emph{Random
  Coordinate Descent} (RCD) ($cf$.~\cite[Section 6.4]{bubeck2015convex}), where at each time step only one
uniformly random coordinate of the gradient is used by the gradient descent algorithm.}

\section{Main results: average information lower bounds for optimization}
  \label{sec:results}
  For $p\in[1,\infty]$ and $D>0$, let
$\mathbb{X}_p(D) \eqdef \{\X \subseteq \R^d: \max_{x,y\in\X}\norm{x-y}_p\leq
D\}$ be the collection of subsets of $\R^d$ whose $\ell_p$ diameter is
at most $D$. In stating our results, we will fix throughout the
parameter $B>0$, the almost sure bound on the gradient magnitude
defined in~\eqref{e:asmp_as_bound}, as well as the strong convexity
parameter $\alpha>0$ defined in~\eqref{e:strong_convexity} (which,
implicitly, is required to satisfy
Lemma~\ref{lemma:B_alpha}). Throughout this section, our lower bounds
on minmax optimization error focus on tracking the convergence rate
for large $T$, \newerest{a standard regime of interest for the stochastic optimization
setting.}

\subsection{Lower bounds for locally private optimization under adaptive gradient processing}
  \label{sec:results:ldp}
 \newerest{Throughout, we consider $\priv\in[0,1]$, namely the high-privacy regime.}
\paragraph{Convex function family.} \newest{For the convex function family, we prove the following lower bounds.}

\begin{thm}\label{t:conpriv_P2}
Let \new{$p\in[1,2)$}, $\priv \in [0, 1]$, and $D>0$. 
There exist absolute constants \new{$c_0,c_1>0$ such that, for
$T \geq c_0  \frac{d}{\priv^2}$},
\[
 \sup_{\X \in \mathbb{X}_p(D)} \ep^*(\X , \oO_{{\tt c}, 1}, T, \W_{{\tt priv}, \priv}) 
\geq 
\frac{c_1DB}{\sqrt{T}} \cdot \sqrt{\frac{d}{\priv^2}}. 
\]
\new{(Moreover, one can take $c_0 \eqdef \frac{1}{2e(e-1)^2}$ and $c_1\eqdef \frac{1}{36(e-1)\sqrt{2e}}$.)}
\end{thm}
\noindent See Section~\ref{s:Proof_conv_P2} for the proof.

\begin{thm}\label{t:conpriv_inf}
\newest{Let $p\in[2,\infty], \priv \in [0, 1]$, and $D>0$. 
There exist absolute constants \new{$c_0,c_1>0$ such that, for}
$T \geq c_0  \frac{d^2}{\priv^2}$},
 \[
 \sup_{\X \in \mathbb{X}_p(D)} \ep^*(\X , \oO_{{\tt c}, p}, T, \W_{{\tt priv}, \priv}) 
\geq  \frac{c_1 DBd^{1/2-1/p}}{\sqrt{T}} \cdot \sqrt{\frac{d}{ \priv^2}}. 
\] 
\new{(Moreover, one can take $c_0$ and $c_1$ as in Theorem~\ref{t:conpriv_P2}.)}
\end{thm}
\noindent See Section~\ref{ssec:con_p_inf} for the proof.

\newerest{
\begin{rem}[Tightness of bounds for convex functions and LDP constraints]
 \cite[Theorem 4 and 5]{duchi2014privacy} 
 provide nonadaptive LDP algorithms which show that Theorem~\ref{t:conpriv_P2} is tight up to logarithmic factors \new{for $p=1$} and Theorem~\ref{t:conpriv_inf} is tight up to constant factors~\new{for all $p\in[2,\infty]$} \new{(to the best of our knowledge, no non-trivial upper bound is known for $p\in(1,2)$.)}.  Therefore, adaptive processing of gradients under LDP cannot significantly improve the convergence rate for convex function families.

Interestingly, for $p=1,$ \cite{duchi2014privacy} also provide a slightly stronger lower bound of $\frac{c_0DB}{\sqrt{T}} \cdot \sqrt{\frac{d \log d}{\priv^2}}$ for nonadaptive protocols, which matches the performance of their nonadaptive protocols up to constant factors. This points to a minor gap in our understanding of adaptive protocols: Can we establish
a stronger lower bound for adaptive protocols to match the performance of the nonadaptive algorithm  of \cite{duchi2014privacy}, or does there exist a better adaptive protocol? \new{We believe that the latter option is correct, and conjecture that the $\sqrt{d \log d}$ dependence is tight even for adaptive protocols.}
\end{rem}}

From Remark~\ref{r:c_p1}, the standard optimization error for $\ell_1$ and $\ell_p$, $p \in [2, \infty]$, convex family blows up by a factor of $\sqrt{d/\priv^2}$ when the gradient estimates are passed through an $\priv$-LDP channel.

\paragraph{Strongly convex family.} \newest{We prove the following result for strongly convex functions.} 
\begin{thm}\label{t:sconpriv_P2} 
Let $\priv\in[0,1]$, and $D>0$. There exist absolute constants \new{$c_0,c_1>0$ such that, for $T \geq c_0\cdot \frac{B^2}{\alpha^2D^2}\cdot\frac{d}{\priv^2}$},
\[
 \sup_{\X \in \mathbb{X}_2(D)} \ep^*(\X , \oO_{{\tt sc}}, T, \W_{{\tt priv}, \priv }) 
\geq 
\frac{c_1 B^2}{\alpha T} \cdot \frac{d}{\priv^2}. 
\]
\end{thm}
\noindent See Section~\ref{ssec:scon} for the proof.

\begin{rem}[Tightness of bounds for strongly convex functions and LDP constraints]
One can use stochastic gradient descent with the nonadaptive protocol from~\cite[Appendix C.2]{duchi2014privacy} to obtain a nonadaptive protocol with
convergence rate
matching the lower bound in Theorem~\ref{t:sconpriv_P2} up to constant factors, establishing that adaptivity does not help for strongly convex functions.
\end{rem}

From Remark~\ref{r:sc}, the standard optimization error  for strongly convex functions
blows up by a factor of $\frac{d}{\priv^2}$ when the gradient estimates are passed through an $\priv$-LDP channel.
\subsection{Lower bounds on communication-constrained optimization}
  \label{sec:results:communication}
\noindent\textbf{Convex function family.} \newest{For convex functions, we prove the following lower bounds.}
\begin{thm}\label{t:concom_P2}
Let \new{$p\in[1,2)$, and $D>0$}. There exists an absolute constant $c_0>0$ such that, for $r  \in \N$, and $T \geq \frac{d}{6r},$ %
\[
 \sup_{\X \in \mathbb{X}_p(D)} \ep^*(\X , \oO_{{\tt c}, 1}, T, \W_{{\tt com}, r  }) 
\geq 
\frac{c_0DB}{\sqrt{T}} \cdot \sqrt{\frac{d}{d \wedge r}} .
\]
\new{(Moreover, one can take $c_0\eqdef \frac{1}{12\sqrt{58}}$.)}
\end{thm}
\noindent See Section~\ref{ssec:scon} for the proof.

\begin{thm}\label{t:concom_inf} %
Let $p \in [2, \infty]$, \new{and $D>0$}. There exists an absolute constant $c_0>0$ such that, for $r \in \N$, and \new{$T \geq \frac{1}{\newer{4}}\cdot\frac{d^2}{2^r\land d}$}, we have
 \[
 \sup_{\X \in \mathbb{X}_p(D)} \ep^*(\X , \oO_{{\tt c}, p}, T, \W_{{\tt com}, r}) 
\geq  \left(\frac{c_0DBd^{1/2-1/p}}{\sqrt{T}} \cdot \sqrt{\frac{d}{ d \wedge 2^r}} \right) \vee \left(  \frac{c_0DB}{\sqrt{T}} \cdot \sqrt{\frac{d}{  d \wedge r}} \right)
\]
\new{(Moreover, one can take $c_0\eqdef \frac{1}{12\sqrt{58}}$.)}  %
\end{thm}
\noindent See Section~\ref{ssec:con_p_inf} for the proof.

\newerest{
\begin{rem}[Tightness of bounds for convex functions and communication constraints]
In Appendix \ref{app:scheme}, we provide a scheme which matches the lower bound in Theorem \ref{t:concom_P2} for $p=1$ up to constant factors for any $r$.
Since each coordinate of oracle output is bounded by $B$ for $p=1$, we simply can use an
unbiased $1$-bit quantizer for each coordinate.
The proposed scheme uses such a quantizer for each
coordinate and makes $d/r$ repeated queries to the oracle for the same point, but
gets $1$-bit information about $r$ different coordinates in each query. 

In general, there are two obstacles in extending this scheme to other cases: First,
the uniform bound of $B$ for each coordinate is too loose. Second, we
cannot assume that repeated queries for the same point give identically distributed
outputs (we only assume that their means are subgradients and they have bounded moments).
We were able to circumvent the second difficulty for $p=1$ using convexity
of the set of subgradients. However, in general, it remains an obstacle. Nonetheless,
if we make \newer{the} assumption \new{that repeated queries yield i.i.d.\ outputs}, we can even attain the lower bound Theorem~\ref{t:concom_inf}
for $p=\infty$ up to a constant factor as follows.
We can use the quantizer \textsf{SimQ} from~\cite{mayekar2020limits} to obtain an unbiased
estimator of the common mean (a subgradient) of the repeated query outputs,
which takes only $d$ distinct values. We can then apply the simulate-and-infer
approach from~\cite{ACT20:IT2} to obtain samples from this $d$-ary distribution
using $r$ bits per query and $O(d/2^r)$ queries per sample. This results in an $O(d/r)$
factor blow-up in the standard convergence rate, which
when used with appropriate mirror descent algorithms
matches our lower
bound in Theorem~\ref{t:concom_inf} for $p=\infty$.

In general, without making any additional assumptions about the oracle,
we can use the quantizer \textsf{SimQ$^+$} from~\cite{mayekar2020limits}
with $k= r$
and appropriate mirror descent algorithms
to get upper bounds that match the
lower bounds in Theorem~\ref{t:concom_inf} for $p\in[2,\infty]$,
up to an additional $O(\log d)$ factor.
For $p=2$, we can use the quantizer \textsf{RATQ} from~\cite{mayekar2020ratq}
to improve this match to an $O(\ln \ln^\ast d)$ factor. \new{However, as was the case in the privacy setting, to the best of our knowledge no non-trivial upper bound is known for $p\in(1,2)$.}
\end{rem}}

\new{From  Remark~\ref{r:c_p1}, the standard optimization errors for $\ell_1$ and $\ell_p$, $p \in [2, \infty]$, convex family blow up by a factor of $\sqrt{\frac{d}{d \wedge r}}$ and $\sqrt{\frac{d}{d \wedge 2^r}} \vee \sqrt{\frac{d^{2/p}}{d \wedge r}} $, respectively, when the gradient estimates are  compressed to $r$ bits.}

\paragraph{Strongly convex family.} \newest{We prove the following result for strongly convex functions.} 
\begin{thm}\label{t:sconcom_P2}
\newer{Let $D>0$.} There exist \new{absolute constants $c_0,c_1>0$} such that, for  $ r  \in \N$ and \new{$T \geq c_0 \cdot \frac{B^2}{\alpha^2D^2}\cdot\frac{d}{r}$},
\[
 \sup_{\X \in \mathbb{X}_2(D)} \ep^*(\X , \oO_{{\tt sc}}, T, \W_{{\tt com}, r }) 
\geq 
\frac{c_1 B^2}{\alpha T} \cdot \frac{d}{d \land r} \,.
\]
\end{thm}
\noindent See Section~\ref{ssec:scon} for the proof.

\begin{rem}[Tightness of bounds for strongly convex functions and communication constraints]
We note that the nonadaptive scheme \textsf{RATQ} in~\cite{mayekar2020ratq} along with stochastic gradient descent matches the lower bound in Theorem~\ref{t:sconcom_P2} up to a \newer{$\ln \ln^\ast d$} factor for $r = \Omega(\ln \ln^\ast d)$. 
\end{rem}

\new{From Remark~\ref{r:sc}, the standard optimization error  for strongly convex functions
blows up by a factor of $\frac{d}{r}$ when the gradient estimates are compressed to $r$ bits.}

\subsection{Lower bounds on computationally-constrained optimization}
  \label{sec:results:rcd}

We restrict to the case of Euclidean geometry ($p=2$) for the oblivious sampling channel family  $\W_{\tt obl}$. Our motivation for introducing this class was to study the optimality of standard RCD, which is proposed to work in the Euclidean setting alone. Furthermore, if we consider a slightly larger family of channels where the sampling probabilities can depend on the input itself, the resulting family will be similar to the 1-bit communication family, which we have addressed in Section \ref{sec:results:communication}.
\paragraph{Convex family.} For convex functions, we establish the following lower bound, for $p=2$.
\begin{thm}\label{t:concompt_P2}
\newer{Let $D>0$.} There exists an absolute constant $c_0>0$ such that, for \new{$T \geq \frac{d}{4}$}, we have
\[
 \sup_{\X \in \mathbb{X}_2(D)} \ep^*(\X , \oO_{{\tt c}, 2}, T, \W_{{\tt obl} }) 
\geq 
\frac{c_0\sqrt{d}DB}{\sqrt{T}}\,.
\]
\new{(Moreover, one can take $c_0 \eqdef \frac{1}{72}$.)}
\end{thm}
\noindent See Section~\ref{s:Proof_conv_P2} for a proof. 

The standard Random Coordinate Descent (RCD) (see for instance~\cite[Theorem 6.6]{bubeck2015convex}), which employs uniform sampling, matches this lower bound up to constant factors. The optimality of standard RCD \newer{motivates further} the folklore approach of uniformly sampling coordinates for random coordinate descent unless there is an obvious structure to exploit (as
in \cite{nesterov2012efficiency}). This establishes that adaptive sampling
strategies do not improve over nonadaptive sampling strategies for the
family $\W_{{\tt obl}}$.
Also from Remark~\ref{r:c_p1}, the standard optimization error for $\ell_2$ convex family blows up by a factor of $\sqrt{d}$ when the gradients are sampled obliviously.
\paragraph{Strongly convex family.} For strongly convex functions, we obtain the following lower bound, for $p=2$.
\begin{thm}\label{t:sconcompt_P2}
\newer{Let $D>0$.} \newest{ There exist \new{absolute constants $c_0,c_1>0$} such that, for \new{$T\geq c_0 \cdot d\frac{B^2}{\alpha^2D^2}$}, we have}
\[
 \sup_{\X \in \mathbb{X}_2(D)} \ep^*(\X , \oO_{{\tt sc}}, T, \W_{{\tt obl} }) 
\geq 
\frac{c_1 dB^2}{\alpha T}\,.
\]
\end{thm}
\noindent See Section~\ref{ssec:scon} for the proof.

\new{Once again, the standard RCD algorithm matches this lower bound, which shows that adaptive sampling strategies do not improve over nonadaptive sampling strategies for strongly convex optimization. Further, from Remark~\ref{r:sc}, the standard optimization error for strongly convex family blows up by a factor of $d$ when the gradients are sampled obliviously.
}

\section{Proofs of average information lower bounds}
  \label{sec:proofs}
  \subsection{Outline of the proof for our lower bounds}
The proofs of our lower bounds for adaptive protocols follow the same general template,
summarized below.

\medskip
{\bf Step 1. Relating optimality gap to average information:}
We consider a family of functions $\G=\{g_v : v\in\{-1,1\}^d\}$
satisfying suitable conditions and associate with it a ``discrepancy metric''
$\psi(\G)$
that allows us to relate the optimality gap of any algorithm to an average mutual
information quantity. Specifically, for $V$ distributed uniformly over $\{-1,1\}^d$,
we show that the output $\hat x$ of any optimization algorithm satisfies
\[
\E{g_{V}(\hat{x})-\min_{x\in \X}g_{V}(x)} \geq { \frac{d\psi(\G)}{6}}  \left[1 - \sqrt{
\frac{\newer{2}}{d}
    \sum_{i=1}^d  \mutualinfo{V(i)}{Y^T}}\right],
\]
\newer{where} $Y_t$ is the channel output for the gradient in the $t$th iteration and $Y^T:=(Y_1, \dots, Y_T)$.

Heuristically, we have related the gap to optimality to the difficulty of inferring $V$ by observing $Y^T$.
We note that the bound above is similar to that of~\cite{agarwal2012information}, but instead of mutual information $\mutualinfo{V}{Y^T}$ we get the average mutual information per coordinate. This latter quantity is amenable to analysis for adaptive protocols.

\smallskip
{\bf Step 2. Average information bounds:} To bound the average mutual information per coordinate,
$\frac{1}{d}\sum_{i=1}^d  \mutualinfo{V(i)}{Y^T}$, we take recourse to the recently proposed bounds  from~\cite{acharya2020general}. These bounds hold for $Y^T$ which is the output of adaptively selected
channels from a fixed channel family $\W$, with i.i.d.\ input $X^T=(X_1,\dots,X_T)$ generated from
a family of distributions $\{\p_v, v\in \{-1,1\}^d\}$. We view the output of oracle as inputs $X^T$
and derive the required bound.

While results in~\cite{acharya2020general} provided bounds for $\W_{{\tt priv}, \priv}$
and $\W_{{\tt comm}, r}$, we extend the approach to handle $\W_{{\tt obl}}$. Specifically, under a smoothness and symmetry condition on $\{\p_v, v\in \{-1,1\}^d\}$, which has a parameter
$\gamma$ associated with it, we show the following:

For $|\X|<\infty$  and $\X_i:=\{x(i):x\in \X\}$, $i\in [d]$,
we have
\[
\sum_{i=1}^{d}\mutualinfo{V(i)}{Y^T} \leq \frac{C}2\cdot T\gamma^2,
\]
where the constant $C$ depends only on $\{\p_v, v\in \{-1,+1\}^d\}$ and, denoting
by $v^{\oplus i}\in \{-1,1\}^d$ the vector with the sign of the $i$th coordinate
of $v$ flipped, is given by
\[
C=(\max_{i\in[d]}|\X_i|-1)\cdot
\max_{x\in \X}\max_{v\in\{-1,+1\}^d} \max_{i\in[d]}\frac{\p_{v^{\oplus i}}(X(i)=x(i))}
         {\p_v(X(i)=x(i))}.
\]

\smallskip
{\bf Step 3. Use appropriate difficult instances} 
\newer{On the one hand,} to prove lower bounds for the convex family we will use the class of functions $\Gconvex=\{g_v(x) \colon v \in \{-1, 1\}^d  \}$ defined on the domain $\X=\{x \in\R^d: \norm{x}_{\infty} \leq b\}$ comprising functions $g_v$ given below:
 \begin{align*}
g_{v}(x) = a \cdot \sum_{i=1}^{d}|x(i)-v(i)\cdot b|, \quad \forall x \in \X, v \in \{ -1, 1\}^d.
 \end{align*}
 On the other hand, to prove lower bounds for the strongly convex family, we will use the class of functions $\Gstrongconvex=\{g_v(x) \colon v \in \{-1, 1\}^d  \}$
 on $\X=\{x \in\R^d: \norm{x}_{\infty} \leq b\}$ given by
 \begin{align*}
g_{v}(x)= a   \sum_{i=1}^{d} \mleft( \frac{1+2\delta v(i)}{2}  f^+_i(x) + \frac{1-2\delta v(i)}{2} f^-_i(x) \mright), \quad \forall x \in \X, v \in \{ -1, 1\}^d, 
 \end{align*}
 where $f^+_i$ and $f^-_i$, for $i\in[d]$, are given by
\begin{align*}
  f^+_i(x) = \theta b |x(i)+b| + \frac{1-\theta}{4}  (x(i)+b)^2, \qquad
  f^-_i(x) &= \theta b |x(i)-b| + \frac{1-\theta}{4}  (x(i)-b)^2.
\end{align*}

\smallskip
{\bf Step 4. Carefully combine everything:} We obtain our
desired bounds by applying Steps 1 and 2 to difficult instances from Step 3. 
Since the difficult instance for convex family consists of linear functions, the gradient
does not depend on $x$. Thus, we can design oracles which give i.i.d.\ output with distribution independent of the query point $x_t$, whereby the bound in Step 2 can be applied. 
Interestingly, we construct different oracles for $p<2$ and $p\geq 2$.

However, the situation is different for the strongly convex family. The gradients now depend on the query point $x_t$, whereby it is unclear if we can comply with the requirements in Step 2. Interestingly, for communication and local privacy constraints, we construct oracles that allow us to view messages $Y^T$ as \newerest{the} output of adaptively selected channels applied to independent samples from a common distribution $\p_v$. While it is unclear if the same can be done for computational constraints as well, we use an alternative approach and exhibit an oracle for which we can find an intermediate message vector $Z_1, \dots, Z_T$ such that \newer{(i)} $V$ and $Y^T$ are conditionally independent given
$Z^T$ and \newer{(ii)} the message $Z^T$ satisfies the requirements of Step 2. 

\subsection{Relating optimality gap to average information}\label{ssec:gap-information}
\new{In this section, we prove a general lower bound for the expected gap
  to optimality by considering a parameterized family of functions and oracles
which is contained in our oracle family of interest.
We present a bound that relates the expected gap to optimality
to the average mutual information between the channel output and
different coordinates of the unknown parameter. 
This step is \newerest{the} key difference between our approach and that of~\cite{agarwal2012information}, which used Fano's method instead
of our bound below. We remark that the bounds resulting from
Fano's method are typically not amenable to analysis for adaptive protocols.
}

In more
detail, our result can be used to prove bounds for the average
optimization error over any class of functions which satisfies the two
conditions below.  
\begin{assum}\label{e:cond_assouad}
Let $\X\subseteq\R^d$ and $\V=\{-1,1\}^d$. Let $\G=\{g_{v}: v \in \V
\}$ where $g_v:\X\to \R$ are real-valued functions from $\X$ such that
\begin{enumerate}
\item the $g_v$s are \emph{coordinate-wise decomposable}, $i.e.$, there exist
  functions $g_{i,b}\colon \R \to \R$, $i\in[d]$, $b\in\{-1,1\}$,
  such that
		\[
			g_{v}(x)=\sum_{i=1}^{d} g_{i,v(i)}(x(i)).
		\]
\item the minimum of $g_v$ is also a \emph{coordinate-wise minimum}, $i.e.$, if we denote by $x^\ast_{v}$ the minimum of $g_{v}$ over $\X$, then, for all $i\in [d]$, we have
		\[
			x^\ast_{v}(i)=\arg\!\min_{y \in \X_i}g_{i,v(i)}(y),
		\]
		where $\X_i =\{x(i) : x \in \X\}$. \label{listref_2}
\end{enumerate} 
\end{assum}
For $\G$ satisfying Assumptions~\ref{e:cond_assouad} and for $i\in[d]$, we now define the following discrepancy metric: 
\begin{align}
  \label{def:psi}
	\psi_{i}(\G) 
		&\eqdef  \min_{y \in \X_i} \mleft( g_{i, 1}(y) +  g_{i, -1}(y) - \mleft(\min_{y'\in\X_i} g_{i, 1}(y') + \min_{y'\in\X_i}g_{i, -1}(y') \mright) \mright)\\
 	\psi(\G) 
		&\eqdef \min_{i \in [d]} \psi _{i}(\G).
\end{align}
This is a ``coordinate-wise counterpart'' of the metric used in~\cite{agarwal2012information}.
The next lemma follows readily from this definition.
\begin{lem}
	\label{lem:useful:consequence:discrepancy}
Fix  $i\in[d]$. For every $y\in \X_i$, there can be at most one $b\in \{-1, 1\}$ such that 
\[
	g_{i, b}(y) - \min_{y'\in\X_i} g_{i, b}(y') \le \frac{\psi_{i}(\G)}{3}.
\]
\end{lem}
\begin{proof} 
Let $b\in\{-1,1\}$. By definition of $\psi_i(\G)$, for all $y \in \X_i$ we have
 \[
 	\mleft( g_{i, b}(y) - \min_{y'\in\X_i} g_{i, b}(y') \mright) + \mleft( g_{i, -b}(y) - \min_{y'\in\X_i} g_{i, -b}(y') \mright) \geq \psi_{i}(\G).
 \]
For $y$ such that $g_{i, b}(y)- \min_{y'\in\X_i} g_{i, b}(y') \leq \frac{\psi_{i}(\G)}{3}$, we  now must have that 
 \[
 	g_{i, -b}(y)- \min_{y'\in\X_i} g_{i, -b}(y') \geq \frac{2\psi_{i}(\G)}{3}.\qedhere
 \]
\end{proof}
We will use this observation to bound the \new{expected gap to optimality} for any algorithm $\pi$ optimizing an unknown function in $\G$ that has access to only the corresponding first-order oracle.

\begin{lem}\label{l:ImpL}
  Suppose $\G=\{g_v : v\in\{-1,1\}^d\}$ satisfies Assumption~\ref{e:cond_assouad}. Let $\pi$ be any optimization algorithm that adaptively selects the channels $\{W_j\}_{j \in [T]}$.
  For a random variable $V$ distributed uniformly over $\{-1,1\}^d$, the output $\hat{x}$ of $\pi$ when it is applied to a function from $\G$ and any associated (stochastic subgradient) oracle
   satisfies
\[
\E{g_{V}(\hat{x})-g_{V}(x_V^\ast)} \geq { \frac{d\psi(\G)}{6}}  \left[1 - \sqrt{
\frac{1}{d}
    \sum_{i=1}^d  \newer{2}\mutualinfo{V(i)}{Y^T}}\right],
\]
where $\psi(\G)=\min_{j \in [d]}\psi_j(\G)$, $Y_t$ is the channel output for the gradient at time step $t$
and $Y^T:=(Y_1, \dots, Y_T)$.
\end{lem} 
\begin{proof}
  \new{Our proof is based on relating the gap to optimality to the error in estimation of $V$
upon observing $Y^T$.
  }
    Suppose the algorithm $\pi$ along with channels $\{W_j\}_{j \in [T]} $
outputs the point $\hat{x}$ after $T$ iterations. By \newer{linearity of
expectation}, the decomposability of $g_v$, and Markov's inequality,
we have
\begin{align}
  \E{g_{V}(\hat{x}) -g_{V}(x_V^\ast) }
  &= \sum_{i=1}^{d} \E{g_{i,
      V(i)}(\hat{x}(i)) -g_{i,V(i)}(x_V^\ast(i))} \notag\\
  &\geq
\sum_{i=1}^{d} \frac{\psi_i(\G)}{3}\bPr{ g_{i, V(i)}(\hat{x}(i))
  -g_{i, V(i)}(x_V^\ast(i)) \geq \frac{\psi_i(\G)}{3} }
\notag
\\
&\geq \frac{\psi(\G)}{3}
\sum_{i=1}^{d} \bPr{ g_{i, V(i)}(\hat{x}(i))
  -g_{i, V(i)}(x_V^\ast(i)) \geq \frac{\psi_i(\G)}{3} }
\label{l:Impl:postMarkov}.
\end{align}
We proceed to bound each summand separately.

Fix any $i\in[d]$ and
consider the following estimate for $V(i)$: Given $\hat{x}$, we
output a $\hat{V}(i) \in \{-1,1\}$ 
satisfying 
 \[
 g_{i,\hat{V}(i)}(\hat{x}(i)) -
\new{\min_{y^\prime\in \X_i}g_{i,\hat{V}(i)}(y^\prime)} < \frac{\psi_i(G)}{3};
 \]
 if no such $\hat{V}(i)$ exists, we generate $\hat{V}(i)$
 uniformly from $\{-1,1\}$.
 Then, as a consequence of Lemma~\ref{lem:useful:consequence:discrepancy}, we get
 \begin{equation}
  \label{eq:bound:proba:error:test:1}
    \bPr{\hat{V}(i)\neq v(i) } \leq \bPr{ g_{i, v(i)}(\hat{x}(i)) -g_{i, v(i)}(x_v^\ast(i)) \geq \frac{\psi_i(\G)}{3} }.
 \end{equation}

\new{Next, denote by $\p^{Y^T}$ the distribution of $Y^T$ and by
  $\p_{+i}^{Y^T}$  and $\p_{-i}^{Y^T}$, respectively, the distributions of $Y^T$
  given $V(i)=+1$ and $V(i)=-1$. It is easy to verify that
  \[
\p^{Y^T} = \frac 12(\p_{+i}^{Y^T}+ \p_{-i}^{Y^T}), \quad \forall\,i\in[d].
  \]} 
 Noting that $V(i)$ is uniform and the estimate $\hat{V}(i)$ is formed as a function of $Y^T$,
 we get
 \begin{equation}
  \label{eq:bound:proba:error:test:2}
  \bPr{\hat{V}(i)\neq v(i) } \geq \frac{1}{2} - \frac{1}{2}
\totalvardist{\p_{+i}^{Y^T}}{\p_{-i}^{Y^T}}.
 \end{equation}
From this, combining~\eqref{eq:bound:proba:error:test:1} and~\eqref{eq:bound:proba:error:test:2} and plugging the result into~\eqref{l:Impl:postMarkov}, we have 
\begin{align*}
\E{g_{v}(\hat{x}) -g_{v}(x_v^\ast) }  
&\geq \frac{\psi(\G)}{6}\sum_{i=1}^{d}
\left[ 1 - \totalvardist{\p_{+i}^{Y^T}}{\p_{-i}^{Y^T}}\right]\\
&\geq \frac{\psi(\G)}{6}\sum_{i=1}^{d}
\left[ 1 - \totalvardist{\p_{+i}^{Y^T}}{\p^{Y^T}}
- \totalvardist{\p_{-i}^{Y^T}}{\p^{Y^T}}
\right]\\
&\geq \frac{\psi(\G)}{6}\sum_{i=1}^{d}
\left[ 1 - \sqrt{\frac {1}2\kldiv{\p_{+i}^{Y^T}}{\p^{Y^T}}}
- \sqrt{\frac {1}2\kldiv{\p_{-i}^{Y^T}}{\p^{Y^T}}|}
\right]
\\
&\geq \frac{d\psi(\G)}{6}
\left[ 1 - \sqrt{\frac {1}{d}\sum_{i=1}^d
  \kldiv{\p_{+i}^{Y^T}}{\p^{Y^T}}
  +\kldiv{\p_{-i}^{Y^T}}{\p^{Y^T}}}
\right]
\\
&= \frac{d\psi(\G)}{6}
\left[ 1 - \sqrt{\frac {\newer{2}}{d}\sum_{i=1}^d
\mutualinfo{V(i)}{Y^T}}
\right],
\end{align*}
where the second inequality follows from the triangle inequality, the third is Pinsker's
inequality, and the fourth is Jensen's inequality. 
\end{proof}
\subsection{Average information bounds}\label{ssec:avg_info}
The next step in our proof is to bound the average mutual information
that emerged in Section~\ref{ssec:gap-information}.
A general recipe for bounding this average mutual information
has been given recently in~\cite{acharya2020general},
which we recall below. 

Let $\{\p_v, v\in \{-1,1\}^d\}$ be a family of distributions over some domain $\X$
and $\W$ be a fixed channel family.
For $v\in\{-1,1\}^d$ and $i\in[d]$, denote by $v^{\oplus i}$ the element of $\{-1,1\}^d$ obtained by flipping the $i$th coordinate of $v$. For a fixed $v$, we obtain $T$ independent samples $X_1, \dots, X_T$ from $\p_v$.
Let $Y_1, \dots, Y_T$ be the output of channels selected from the channel family $\W$ by
an adaptive channel selection strategy (see Section~\ref{ssec:set-up}) when input to the channel at
time $t$ is $X_t$, $1\leq t\leq T$.\footnote{The bound in~\cite{acharya2020general} allows
  even shared randomness $U$ in its definition of interactive protocols. We have omitted
  $U$ in this paper for simplicity.}

For $V$ distributed uniformly on $\{-1,1\}^d$, we are interested in bounding $(1/d)\sum_{i=1}^d\mutualinfo{V(i)}{Y^{T}}$.
In~\cite{acharya2020general}, different bounds were given for this quantity under different assumptions. We state these assumptions below.
\begin{assum}
  \label{assn:decomposition-by-coordinates}
For every $v\in\{-1,1\}^d$ and $i\in[d]$,
there exists $\phi_{v,i}\colon\X\to\R$
such that $\EE{\p_v}{\phi_{v,i}^2}=1$,
$\EE{\p_{v}}{\phi_{v,i}\phi_{v,j}}=\indic{\{i=j\}}$ holds
for all $i,j\in[d]$,
and
\[
 \frac{d\p_{v^{\oplus i}}}{d\p_v}=1+\gamma\phi_{v,i}
,
 \]
 where $\gamma\in \R$ is a fixed constant independent of $v,i$.
\end{assum}
\begin{assum}
  \label{assn:ratio:bounded}
  There exists some $\kappa_{\W} \geq 1$ such that
\[
\max_{v\in \{-1,1\}^d} \max_{y\in \cY} \sup_{W\in \W} \frac{\EE{\p_{v^{\oplus i}}}{W(y\mid X)}}{\EE{\p_{v}}{W(y\mid X)}}
\leq \kappa_{\W}.
\]
\end{assum}
\begin{assum}\label{assn:subgaussianity}
There exists some $\sigma\geq 0$ such
that, for all $v\in\{-1,1\}^d$, the vector $\phi_v(X)\eqdef
(\phi_{v,i}(X))_{i\in[d]}\in\R^d$ is $\sigma^2$-subgaussian
for $X\sim\p_v$.\footnote{Recall that a random variable $Y$ is
$\sigma^2$-subgaussian if $\E{Y}=0$ and $\E{e^{\lambda Y}}\leq
e^{\sigma^2\lambda^2/2}$ for all $\lambda\in\R$; and that a
vector-valued random variable $Y$ is $\sigma^2$-subgaussian if its
projection $\langle{Y,u}\rangle$ is $\sigma^2$-subgaussian for every unit
vector $u$.} Further, for any fixed $z$, the random variables
$\phi_{v,i}(X)$ are independent across $i\in[d]$.
\end{assum}
We then have the following bound local privacy constraints.
\begin{thm}[{\cite[Corollary 6]{acharya2020general}}]
\label{cor:ldp}
Consider $\{\p_v, v\in \{-1,1\}^d\}$ satisfying Assumption~\ref{assn:decomposition-by-coordinates}
and the channel family $\W=\W_{{\tt priv}, \priv}$. Let $V$
be distributed uniformly over $\{-1,1\}^d$
and $Y^T$ be the output of channels selected by
the optimization algorithm
as above.
Then, we have
\[
\sum_{i=1}^{d}\mutualinfo{V(i)}{Y^T} \leq T \cdot \frac{\gamma^2}{2} \cdot e^\priv(e^\priv-1)^2.
\]
\end{thm}
\noindent For the case of communication constraints, we have the analogous statement below:
\begin{thm}[{\cite[Corollary 6]{acharya2020general}}]
  \label{cor:simple-numbits}
  Consider $\{\p_v, v\in \{-1,1\}^d\}$ satisfying Assumptions~\ref{assn:decomposition-by-coordinates}
  and~\ref{assn:ratio:bounded}
  and the channel family $\W=\W_{{\tt com}, r}$.
Let $V$
be distributed uniformly over $\{-1,1\}^d$
and $Y^T$ be the output of channels selected by
the optimization algorithm
as above.
Then, we have
\[
\sum_{i=1}^{d}\mutualinfo{V(i)}{Y^T}  \leq \frac{1}{2}\kappa_{\W_{{\tt com}, r}} \cdot T\gamma^2 (2^r\land d).
\]
Moreover, if Assumption~\ref{assn:subgaussianity} holds as well, we have
\[
\sum_{i=1}^{d}\mutualinfo{V(i)}{Y^T} \leq (\ln 2)\kappa_{\W_{{\tt com}, r}}\,\sigma^2 \cdot T\gamma^2 r.
\]
\end{thm}
Finally, we derive a bound for the oblivious sampling channel family.
\begin{thm}
\label{cor:obl}
Consider $\{\p_v, v\in \{-1,1\}^d\}$ satisfying Assumption~\ref{assn:decomposition-by-coordinates}
and the channel family $\W=\W_{{\tt obl}}$.
Let $V$
be distributed uniformly over $\{-1,1\}^d$
and $Y^T$ be the output of channels selected by
the optimization algorithm
as above.
Further, assume that $|\X|<\infty$. 
Then, we have
\[
\sum_{i=1}^{d}\mutualinfo{V(i)}{Y^T} \leq \frac{C}2\cdot T\gamma^2,
\]
where the constant $C$ depends only on $\{\p_v, v\in \{-1,+1\}^d\}$ and, denoting
$\X_i:=\{x(i):x\in \X\}$, is given by
\[
C=(\max_{i\in[d]}|\X_i|-1)\cdot
\max_{x\in \X}\max_{v\in\{-1,+1\}^d} \max_{i\in[d]}\frac{\p_{v^{\oplus i}}(X(i)=x(i))}
         {\p_v(X(i)=x(i))}.
\]
\end{thm}
\begin{proof}
We recall another result from~\cite[Theorem 5]{acharya2020general}:
Under Assumptions~\ref{assn:decomposition-by-coordinates}
  and~\ref{assn:ratio:bounded}, we have\footnote{This is the general bound underlying Theorem~\ref{cor:ldp}.}
\[
\sum_{i=1}^{d} \mutualinfo{V(i)}{Y^T} \leq \frac{1}{2}\kappa_{\W_{\tt obl}}  \cdot T\gamma^2 \max_{v 
\in \{-1, 1\}^d}\max_{W\in\W_{\tt obl}} \sum_{y\in\Y} \frac{\operatorname{Var}_{\p_{v}}[W(y\mid X)]}{\EE{\p_{v}}{W(y\mid X)}}.
\]
We now evaluate various parameters involved in this bound.
  Let $W$ be a oblivious sampling channel specified by the probability vector $(p_i)_{i \in [d]}.$ 
Note that a channel $W\in \W_{{\tt obl}}$ can be equivalently viewed as having output alphabet
$\Y=\{(i, z)\colon z\in \X_i, i\in [d] \}$.
Recall that for an input $x$, the channel output is $x(i)$ with probability $p_i$, $i\in[d]$,
$i.e.$, for $y=(i,z)$, $W(y\mid x) =p_i\indic{\{x(i)=z\}}$.
Thus, we have
\begin{align*}
  \sum_{y\in\Y} \frac{\operatorname{Var}_{\p_{v}}[W(y\mid X)]}{\EE{\p_{v}}{W(y\mid X)}}
  &= \sum_{i=1}^d\sum_{z\in\X_i}\frac{p_i^2\bPr{X(i)=z}-p_i^2\bPr{X(i)=z}^2}{p_i\bPr{X(i)=z}}
  \\
  &= \sum_{i=1}^dp_i(|\X_i|-1)
  \\
  &\leq \max_{i\in[d]}|\X_i|-1.
\end{align*}
Furthermore, proceeding similarly, we get that Assumption~\ref{assn:ratio:bounded} holds as well
with
\[
\kappa_{\W_{\tt obl}}=\max_{x\in \X}\max_{v\in\{-1,+1\}^d} \max_{i\in[d]}\frac{\p_{v^{\oplus i}}(X(i)=x(i))}
         {\p_v(X(i)=x(i))}.
\]
The proof is completed by combining the bounds above.
\end{proof}

\subsection{The difficult instances for our lower bounds}
With our general tools ready, we now describe the precise constructions of function families
we use to get our lower bounds. 
We first provide the details of a family $\Gconvex(a,b)$ of convex functions, before turning to $\Gstrongconvex(a,b,\delta,\theta)$, our family of hard instances for the strongly convex setting. In both cases, our families of hard instances are parameterized (by $a,b$ and $a,b,\delta,\theta$, respectively), and setting those parameters carefully will enable us to prove our various results.

\paragraph{Difficult functions for the convex family.}
To prove lower bounds for the convex family, we will use the class of functions $\Gconvex(a,b)$ below, parameterized by $a,b>0$ and defined on the domain $\X$ as follows:
 \begin{align}\label{e:conv_bottl}
 \nonumber & \X=\{x \in\R^d: \norm{x}_{\infty} \leq b\},\\
\nonumber &g_{v}(x) = a \cdot \sum_{i=1}^{d}|x(i)-v(i)\cdot b|, \quad \forall x \in \X, v \in \{ -1, 1\}^d, \text{ and} \\
&\Gconvex=\{g_v(x) \colon v \in \{-1, 1\}^d  \}.
 \end{align}
 
 Observe that the class $\Gconvex$ satisfies the conditions in Assumption~\ref{e:cond_assouad} with $g_{i,1}(x) = a|x(i)- b|$ and $g_{i,-1}(x)=a|x(i)+b|$ and $\X_i = [-b, b]$ for all $i \in [d].$
Further, we can bound the discreprency metric for this class as follows.
\begin{lem}\label{l:psiGc}
For the class of functions $\Gconvex$ defined in \eqref{e:conv_bottl}, we have
$\psi(\Gconvex) \geq 2 a b$.
\end{lem}
\begin{proof}
  Note that $\min_{x \in [-b, b]}g_{i, 1}(x)=\min_{x \in [-b, b]}g_{i, -1}(x)=0$. Therefore, for all $i\in[d]$, 
\[
  \psi_i(\Gconvex) = \min_{x \in [-b, b]}  \left( a |x(i)- b| +a|x(i)+ b| \right) \geq 2 a b,
\]
where the inequality follows from the triangle inequality.
\end{proof}

\paragraph{Difficult functions for the strongly convex family.}
To prove lower bounds for the strongly convex family, we will use the class of functions $\Gstrongconvex(a,b,\delta,\theta)$, parameterized by $a,b>0$, $\delta>0$, and $\theta\in[0,1]$, and defined on the domain $\X$ as follows:
 \begin{align}\label{e:sconv_bottl}
 \nonumber & \X=\{x \in\R^d: \norm{x}_{\infty} \leq b\},\\
\nonumber &g_{v}(x)= a   \sum_{i=1}^{d} \mleft( \frac{1+2\delta v(i)}{2}  f^+_i(x) + \frac{1-2\delta v(i)}{2} f^-_i(x) \mright), \quad \forall x \in \X, v \in \{ -1, 1\}^d, \text{ and} \\
&\Gstrongconvex=\{g_v(x) \colon v \in \{-1, 1\}^d  \},
 \end{align}
 where $f^+_i$ and $f^-_i$, for $i\in[d]$, are given by
\begin{align}\label{e:f+f_}
  f^+_i(x) &= \theta b |x(i)+b| + \frac{1-\theta}{4}  (x(i)+b)^2,  \\
  f^-_i(x) &= \theta b |x(i)-b| + \frac{1-\theta}{4}  (x(i)-b)^2,
\end{align}
for all $x\in\X$. We can check that, for every $v\in\{-1,1\}^d$, the function $g_v$ is then $\alpha$-strongly convex for $\alpha \eqdef a\cdot \frac{1-\theta}{4}$. Moreover, we have the following bound for the discrepancy metric.
\begin{lem}\label{l:psiGsc}
For the class of functions $\Gstrongconvex$ defined in \eqref{e:sconv_bottl}, if $\frac{1-\theta}{1+\theta}\geq 2 \delta$ then
$
\psi(\Gstrongconvex) \geq \frac{2 a b^2 \delta^2  }{1-\theta}\,.
$
\end{lem}
\begin{proof}
This follows from similar calculations as in \cite[Appendix A]{agarwal2012information}; \newer{we provide the proof here for completeness. Fixing any $v\in\{-1,1\}^d$, we first note that by definition of $\Gstrongconvex$, the function $g_v$ can be indeed be decomposed as
$
  g_v(x) = \sum_{i=1}^d g_{i,v(i)}(x_i)
$ 
for $x\in\cX$ (i.e., $\norm{x}_\infty\leq b$), where, for $i\in[d]$, $\nu\in\{-1,1\}$ and $y \in \cX_i \eqdef [-b,b]$,
\begin{align*}
    g_{i,\nu}(y) 
    &= a\mleft(  \frac{1+2\delta \nu}{2} \mleft( \theta b |y+b| + \frac{1-\theta}{4}(y+b)^2 \mright) + \frac{1-2\delta \nu}{2} \mleft( \theta b |y-b| + \frac{1-\theta}{4}(y-b)^2 \mright) \mright) \\
    &= a\mleft(  \frac{1-\theta}{4}y^2 + \frac{1+3\theta}{4}b^2 + \delta\nu (1+\theta)by \mright)
\end{align*}
where the second line relies on the fact that $|y+b|=y+b$ and $|y-b|=b-y$ for $|y|\leq b$. One can easily see, e.g., by differentiation, that $g_{i,\nu}$ is minimized at $y^\ast \eqdef -2\delta\nu\frac{1+\theta}{1-\theta}b$ which does satisfy $|y^\ast|\leq b$ given our assumption $\frac{1-\theta}{1+\theta}\geq 2 \delta$. It follows that $\min_{y\in\cX_i} g_{i,1}(y) = \min_{y\in\cX_i} g_{i,-1}(y) = ab^2 \mleft( \frac{1+3\theta}{4} - \delta^2\frac{(1+\theta)^2}{1-\theta} \mright)$.
Similarly, we have, for $y\in\cX_i$,
\begin{align*}
    g_{i,1}(y) + g_{i,-1}(y) 
    &= a\mleft(  \frac{1-\theta}{2}y^2 + \frac{1+3\theta}{2}b^2 \mright)
\end{align*}
which is minimized at $y^\ast=0$, where it takes value $ab^2\frac{1+3\theta}{2}$. Putting it together,
\[
  \psi_i(\Gstrongconvex) = \min_{y\in\cX_i} ( g_{i,1}(y) + g_{i,-1}(y) ) - (\min_{y\in\cX_i} g_{i,1}(y) + \min_{y\in\cX_i} g_{i,-1}(y) ) = 2ab^2\delta^2\frac{(1+\theta)^2}{1-\theta}\,.
\]
Finally, $\psi(\Gstrongconvex) = \min_{i\in[d]} \psi_i(\Gstrongconvex) = 2ab^2\delta^2\frac{(1+\theta)^2}{1-\theta} \geq \frac{2ab^2\delta^2}{1-\theta}$, as claimed.
}
\end{proof}

\subsection{Convex Lipschitz functions for $p\in[1,2)$: Proof of Theorems \ref{t:conpriv_P2},   \ref{t:concom_P2}, and \ref{t:concompt_P2}}\label{s:Proof_conv_P2}
We first prove Theorems \ref{t:conpriv_P2} and \ref{t:concom_P2}, our lower bounds on optimization of convex functions for $p\in[1,2)$ under privacy and communication constraints, respectively. We consider the class of functions $\G_{\tt c}$ defined in \eqref{e:conv_bottl} with parameters $a\eqdef 2B \delta/d^{1/q}$ and $b \eqdef  D/(2d^{1/p})$. That is, $\X=\{x \in \R^d:\norm{x}_\infty\leq D/(2d^{1/p}) \}$ and 
\begin{equation}
  \label{eq:def:gv:convex:p2}
g_{v}(x) \eqdef \frac {2B\delta}{d^{1/q}} \sum_{i=1}^{d} \mleft|x(i) - \frac{v(i) D}{2 d^{1/p}}\mright|
\qquad x \in \X, v \in \{-1, 1 \}^d.
\end{equation}
Note that the gradient of $g_v$ is equal to $-2B\delta v /d^{1/q}$ at every $x\in \X$.

For each $g_v$, consider the corresponding gradient oracle $O_v$ which outputs 
independent values for each coordinate,
with the $i$th coordinate taking values $-B/d^{1/q}$ and $B/d^{1/q}$ with probabilities $(1+2\delta v(i))/2$ and $(1-2\delta v(i))/2$, respectively, for some parameter $\delta>0$ to be suitably chosen later.

Clearly, $\X \in \mathbb{X}_p(D)$ and all  the functions $g_v$ and the corresponding  oracles $O_v$ belong to the convex function family $\oO_{{\tt c}, p}$.
We begin by noting that for $V$ distributed uniformly over $\{-1,1\}^d$, we have
\[ \sup_{\X \in \mathbb{X}_p(D)} \ep^*(\X , \oO_{{\tt c}, p}, T, \W_{{\tt priv}, \priv}) 
\geq \E{g_{V}(x_T)-g_{V}(x_V^\ast)},  \]
where the expectation is over $v$ as well as the randomness in $x_T$.

From Lemma~\ref{l:ImpL} and \ref{l:psiGc}, we have
\begin{equation}\label{e:C_P2}
  \E{g_{V}(x_T)-g_{V}(x_V^\ast)} \geq {\frac{d\cdot a b}{3} \cdot}  \left[1 -  \sqrt{
\frac{\newer{2}}{d}
      \sum_{i=1}^d \mutualinfo{V(i)}{Y^T}}\right],
\end{equation}
where $Y^T=(Y_1, ..., Y_T) $ are the channel outputs for the gradient estimates supplied by the oracle for the $T$ queries. 

Next, we apply the average information bound from Section~\ref{ssec:avg_info}.
To do so, observe that by the definition of our oracle,
the oracle output at each time step is an independent draw from the 
product distribution $\p_v$ on $\Omega \eqdef \mleft\{-\frac{B}{d^{1/q}},\frac{B}{d^{1/q}}\mright\}^d$ (in particular, $\p_v$ is the same at each time step, as it does not depend on the query $x_t$ at time step $t$ to the oracle). We treat the output of the independent outputs of the oracle as i.i.d. samples $X_1, ...,X_T$
in Section~\ref{ssec:avg_info} and the corresponding channel outputs as $Y^T$.
We can check that, for every $i\in[d]$, we have 
\begin{equation}
\label{eq:convex:p12:ratio}
    \frac{\p_{v^{\oplus i}}(x)}{\p_v(x)} = \frac{1+2\delta v(i) \sign(x(i))}{1-2\delta v(i) \sign(x(i))}
\end{equation}
for all $x\in\Omega$, and that Assumption~\ref{assn:decomposition-by-coordinates} is satisfied with
\begin{equation}
\label{eq:convex:p12:gamma:phi}
  \gamma \eqdef \frac{4\delta}{\sqrt{1-4\delta^2}}, \qquad \phi_{i,v}(x) \eqdef \frac{v(i) \sign(x(i))+2\delta}{\sqrt{1-4\delta^2}}\,.
\end{equation}
Furthermore, noting that Assumption~\ref{assn:ratio:bounded} always holds with
\[
\kappa_\W= \max_{v\in \{-1,1\}^d}\max_{x\in \Omega}\max_{i\in [d]}
    \frac{\p_{v^{\oplus i}}(x)}{\p_v(x)},
\]
it is satisfied with $\kappa_{\W} = 2$  (regardless of $\W$), as long as $\delta \leq 1/6$,
since the right-side above is bounded by $2$ for such a $\delta$. Finally, Assumption~\ref{assn:subgaussianity}, is also satisfied as $(\phi_{i,v}(X))_{i\in[d]}$ for $X\sim\p_v$ is $\sigma^2$-subgaussian for $\sigma^2 \eqdef \frac{1}{1-4\delta^2}$. 

\paragraph{Completing the proof of Theorem \ref{t:conpriv_P2} (LDP constraints).}
From Theorem~\ref{cor:ldp} and the bounds derived above, we have 
\[
  \sum_{i=1}^d \mutualinfo{V(i)}{Y^T} \leq T\cdot  \frac{8\delta^2}{1-4\delta^2}\cdot e^\priv(e^\priv-1)^2,
\]
and therefore,
\[
  \sum_{i=1}^d \mutualinfo{V(i)}{Y^T} 
\leq c\cdot T \delta^2 \priv^2,
\]
where $c\eqdef 9e(e-1)^2$ (recalling that $\priv\in(0,1]$ and $\delta \leq 1/6$). 
Substituting this bound on the average mutual information in \eqref{e:C_P2} along with the values of $a$ and $b$, we have
\[
\E{g_{V}(x_T)-g_{V}(x_V^\ast)} \geq {\frac{DB\delta}{3} \cdot}  \left[1 -  \sqrt{
    \frac{\newer{2}c T\delta^2 \priv^2}d}\right].
\]
Upon setting $\delta \eqdef \sqrt{\frac{d}{\newer{8}cT\priv^2}}$, we get
\[
\E{g_{V}(x_T) -g_{V}(x_V^\ast) } 
\geq \frac{1}{12\sqrt{\newer{2}c}}\cdot\frac{DB}{\sqrt{T}}\cdot \sqrt{\frac{d}{\priv^2}},
\]
where we require $T \geq \frac{9}{\newer{2}c}\cdot\frac{d}{\priv^2}$ in order to enforce $\delta \leq 1/6$.
\qed 

\paragraph{Completing the proof of Theorem \ref{t:concom_P2} (Communication constraints).}
From Theorem~\ref{cor:simple-numbits} and $\gamma$, $\sigma$, and $\kappa_\W$ set as discussed above, we have 
\[
  \sum_{i=1}^d\mutualinfo{V(i)}{Y^T} \leq \frac{32(\ln 2)}{(1-4\delta^2)^2}\cdot T \delta^2 r,
\]
whereby, using $\delta\leq 1/6$,
\[
\sum_{i=1}^d\mutualinfo{V(i)}{Y^T}
\leq 29 T \delta^2 r.
\]
Substituting this bound on mutual information in \eqref{e:C_P2} along with the values of $a$ and $b$, we have
\[
\E{g_{V}(x_T)-g_{V}(x_V^\ast)} \geq \frac{DB\delta}{3} \cdot  \left[1 -  \frac{1}{\sqrt{d}}\cdot\sqrt{\newer{58} T\delta^2 r}\right].
\]
Setting $\delta\eqdef \sqrt{\frac{d}{\newer{232} r T}}$, we finally get
\[
\E{g_{V}(x_T) -g_{V}(x_V^\ast) } \geq \frac{1}{12\sqrt{\newer{58}}}\cdot \frac{DB}{\sqrt{T}}\cdot \sqrt{\frac{d}{r}},
\]
where we require $T\geq \frac{9}{ \newer{58}}\cdot \frac{d}{r}$ in order to enforce $\delta \leq 1/6$.
\qed

\paragraph{Completing the proof of Theorem \ref{t:concompt_P2} (Computational constraints).}
Note that the sets $\X_i$s in Theorem~\ref{cor:obl} have $|\X_i|=2$ for our oracle.
Further,
\[
\frac{\p_{v^{\oplus i}}(X(i)=x(i))}{\p_v(X(i)=x(i))} =
\frac{\p_{v^{\oplus i}}(x)}{\p_v(x)}=
\frac{1+2\delta v(i) \sign(x(i))}{1-2\delta v(i) \sign(x(i))}\leq 2,
\]
when $\delta\leq 1/6$. Thus, the constant $C$
 in Theorem~\ref{cor:obl} is less than $2$, whereby
\[
\sum_{i=1}^d\mutualinfo{V(i)}{Y^T} \leq  \frac{16 \delta^2}{1-4\delta^2}\cdot T ,
\]
whereby, using $\delta\leq 1/6$,
\[
\sum_{i=1}^d\mutualinfo{V(i)}{Y^T}  \leq 18T\delta^2.
\]
Substituting this bound on mutual information in \eqref{e:C_P2} along with the values of $a$ and $b$, we have
\[
\E{g_{V}(x_T)-g_{V}(x_V^\ast)} \geq \frac{DB\delta}{3} \cdot  \left[1 -  \frac{1}{\sqrt{d}}\cdot\sqrt{\newer{36} T\delta^2 }\right].
\]
Setting $\delta\eqdef \sqrt{\frac{d}{\newer{144}T}}$, we finally get
\[
\E{g_{V}(x_T) -g_{V}(x_V^\ast) } \geq \frac{1}{\newer{72}}\cdot \frac{DB\sqrt{d}}{\sqrt{T}},
\]
where we require $T\geq \frac{d}{\newer{4}}$ in order to enforce $\delta \leq 1/6$.
\subsection{Convex Lipschitz functions for $p\in [2,\infty]$: Proof of Theorems \ref{t:conpriv_inf} and \ref{t:concom_inf}}\label{ssec:con_p_inf}
Next, we establish Theorems \ref{t:conpriv_inf} and \ref{t:concom_inf}, the analogous lower bounds on optimization of convex functions when $p\in[2,\infty)$. We again consider the class of functions $\G_{\tt c}$ defined in \eqref{e:conv_bottl}, this time with parameters $a\eqdef 2B \delta/d$ and $b \eqdef  D/(2d^{1/p})$ That is, here $\X=\{x:\norm{x}_\infty\leq D/(2d^{1/p}) \}$ and 
 \[
g_{v}(x) \eqdef \frac {2B\delta}{d} \sum_{i=1}^{d} \left|x(i) - \frac{v(i) D}{2 d^{1/p}}\right|.
\quad \forall x \in \X, v \in \{-1, 1 \}^d.
\]
It follows that the gradient of $g_v$ is equal to $-2B\delta v /d$ at every $x\in \X$.

For each $g_v$, consider then the gradient oracle $O_v$ which
outputs $0$ in all but a randomly chosen coordinate; if that coordinate is $i$,
it takes values $-B$ and $B$ with probabilities $\frac{1+2\delta v(i))}{2d}$ and $\frac{1-2\delta v(i)}{2d}$, respectively, for some parameter $\delta\in(0,1/6]$ to be suitably chosen later. Thus, the oracle is no longer
a product distribution.

Clearly, $\X \in \mathbb{X}_p(D)$ and all  the functions $g_v$ and the corresponding  oracles $O_v$ belong to the convex function family $\oO_{{\tt c}, p}$. Proceeding as in Section \ref{s:Proof_conv_P2}, we get
for a uniformly distributed $V$ that
\begin{equation}
	\label{e:C_Pinf}
	\E{g_{V}(x_T)-g_{V}(x_V^\ast)} \geq { \frac{DB \delta}{3d^{1/p}} \cdot}  \left[1 -  \sqrt{
\frac 1d
            \sum_{i=1}^d \newer{2}\mutualinfo{V(i)}{Y^T}}\right] .
\end{equation}
Further, proceeding as in the previous section to bound the average information,
we note that the oracle outputs independent samples from the \newer{distribution $\p_v$ on $\Omega \eqdef \mleft\{-B,0,B\mright\}^d$}  %
 at each time. It can be checked  easily that, for every $i\in[d]$, the expression of the ratio $\frac{\p_{v^{\oplus i}}}{p_v}$ given in~\eqref{eq:convex:p12:ratio} still holds (as only the denominators of the Bernoulli parameters have changed, and they cancel out in the ratio), and that Assumption~\ref{assn:decomposition-by-coordinates} is satisfied with the following $\gamma$, $\phi_{i,v}$s:
\begin{equation}
	\label{eq:convex:pinf:gamma:phi}
  \gamma \eqdef \frac{1}{\sqrt{d}}\cdot\frac{4\delta}{\sqrt{1-4\delta^2}}, \qquad \phi_{i,v}(x) \eqdef \sqrt{d}\cdot\frac{v(i) \sign(x(i))+2\delta}{\sqrt{1-4\delta^2}}\,.
\end{equation}
Observe the difference with the expressions from the previous section (specifically,~\eqref{eq:convex:p12:gamma:phi}), as the orthonormality assumption now crucially introduces a factor $1/\sqrt{d}$ in the value of $\gamma$. Finally, because we will enforce $\delta \leq 1/6$ we also can take $\kappa_{\W_{{\tt com}, r}} = 2$ for the communication constraints, as before. We remark that $\phi_{i,v}(X)$ is no longer subgaussian.

\paragraph{Completing the proof of Theorem \ref{t:conpriv_inf} (LDP constraints).}
From Theorem~\ref{cor:ldp} and the value of $\gamma$ above, we get, analogously to the previous section,  
\[
\sum_{i=1}^d \mutualinfo{V(i)}{Y^T} \leq
c\cdot \frac{T \delta^2 \priv^2}{d},
\]
where $c\eqdef 9e(e-1)^2$ (recalling that $\priv\in(0,1]$ and $\delta \leq 1/6$). 
Substituting this bound on mutual information in \eqref{e:C_Pinf}, we obtain
\[
  \E{g_{V}(x_T)-g_{V}(x_V^\ast)} \geq \frac{DB\delta}{3d^{1/p}} \,  \left[1 -  \sqrt{\frac{\newer{2}c T\delta^2 \priv^2}{d^2}}\right].
\]
Optimizing over $\delta$, we set $\delta \eqdef \sqrt{\frac{d^2}{8cT\priv^2}}$ and get
\[
\E{g_{V}(x_T) -g_{V}(x_V^\ast) }
\geq \frac{1}{12\sqrt{\newer{2}c}}\cdot\frac{DBd^{1/2-1/p}}{\sqrt{T}}\cdot \sqrt{\frac{d}{\priv^2}},
\]
where we require $T \geq \frac{9}{\newer{2}c}\cdot\frac{d^2}{\priv^2}$ in order to guarantee $\delta \leq 1/6$. This concludes the proof.
\qed 

\paragraph{Completing the proof of Theorem \ref{t:concom_inf} (Communication constraints).}
We prove the two parts of the lower bounds separately, starting with the first. From Theorem~\ref{cor:simple-numbits} and the setting of $\gamma$ and $\kappa_\W$ as above, we have 
\[
\sum_{i=1}^d \mutualinfo{V(i)}{Y^T}
  \leq \frac{16}{1-4\delta^2}\cdot T \delta^2 \frac{2^r\land d}{d},
\]
whereby, using $\delta\leq 1/6$,
\[
\sum_{i=1}^d \mutualinfo{V(i)}{Y^T}\leq 18 T \delta^2 \frac{2^r\land d}{d}.
\]
Substituting this bound on mutual information in \eqref{e:C_Pinf}, we have
\[
\E{g_{V}(x_T)-g_{V}(x_V^\ast)}
\geq \frac{DB\delta}{3d^{1/p}} \cdot  \left[1 -  \frac{1}{\sqrt{d}}\cdot\sqrt{\newer{36} T\delta^2\frac{2^r\land d}{d}}\right].
\]
Setting $\delta\eqdef \sqrt{\frac{d^2}{\newer{144}(2^r\land d)T}}$, we finally get
\[
\E{g_{V}(x_T) -g_{V}(x_V^\ast) } \geq \frac{1}{\newer{72}}\cdot \frac{DB d^{1/2-1/p}}{\sqrt{T}}\cdot \sqrt{\frac{d}{2^r\land d}},
\]
where we require $T\geq \frac{1}{\newer{4}}\cdot \frac{d^2}{2^r\land d}$ in order to guarantee $\delta \leq 1/6$.

The second bound follows by noting that the lower bound in Theorem \ref{t:concom_P2} is still valid. Finally, since $\frac{d^2}{2^r\land d} \geq \frac{d}{r}$ for all $1\leq r\leq d$, both bounds apply whenever $T = \Omega\mleft(\frac{d^2}{2^r\land d}\mright)$, as claimed.
\qed

\subsection{Strongly convex functions: Proof of Theorem \ref{t:sconpriv_P2}, \ref{t:sconcom_P2}, and \ref{t:sconcompt_P2}}\label{ssec:scon}
Next, we establish our lower bounds on strongly convex optimization. We consider the class of functions $\G_{\tt sc}$ defined in \eqref{e:sconv_bottl} with parameters $a\eqdef B /(\sqrt{d} b)$ and $b \eqdef D/(2\sqrt{d}).$ That is, $\X=\{x:\norm{x}_\infty\leq D/(2 \sqrt{d}) \}$, and, for every $x \in \X$ and $v \in \{ -1, 1\}^d$,
 \[
	g_{v}(x) \eqdef \frac {B}{b\cdot \sqrt{d}}  \sum_{i=1}^{d} \frac{1+2\delta v(i)}{2} f^+_i(x) + \frac{1-2\delta v(i)}{2}f^-_i(x),
\]
and
\[
	f^+_i(x)= \theta b |x(i)+b| + \frac{1-\theta}{4}  (x(i)+b)^2 ~ \text{and} ~ f^-_i(x)= \theta b |x(i)-b| + \frac{1-\theta}{4}  (x(i)-b)^2.
\]
Moreover, in order to ensure that the every $g_v$ is $\alpha$-strongly convex, we choose $\theta\eqdef 1-\frac{4\alpha}{a}$ (so that $a\frac{1-\theta}{4} =\alpha$). It remains to specify $\delta$, which we will choose such that $0< \delta \leq \frac{1}{2}\cdot \frac{1-\theta}{1+\theta}$ in the course of the proof.

For each $g_v$, consider the gradient oracle $O_v$ which on query $x$ outputs independent values for each coordinate,
with the $i$th coordinate taking values
$\frac{B}{b \sqrt{d}}\cdot \frac{\partial f^+_i(x)}{\partial x_i}$ 
and 
$\frac{B}{b \sqrt{d}}\cdot \frac{\partial f^-_i(x)}{\partial x_i}$ 
with probabilities 
$\frac{1+2\delta v(i))}{2}$
 and 
 $\frac{1-2\delta v(i)}{2}$, 
respectively. 

Note that we have $\left| \frac{\partial f^+_i(x)}{\partial x_i} \right|, \left| \frac{\partial f^-_i(x)}{\partial x_i} \right| \leq b$ for all $x$ and $i$, and therefore the gradient estimate $\hat{g}(x)$ supplied by the oracle $O_v$ at $x$ satisfies $\norm{\hat{g}(x)}_2^2\leq B^2$ with probability one, for every query $x\in\X$. 
Further, it is clear that $\X \in \mathbb{X}_2(D)$ and all the functions $g_v$ and the corresponding  oracles $O_v$ belong to the strongly convex function family $\oO_{{\tt sc}}$.

Using our assumption that $\delta \leq \frac{1}{2}\cdot \frac{1-\theta}{1+\theta}$, we obtain by Lemma \ref{l:psiGsc}
\begin{equation}
	\label{eq:inequality:psiGsc}
	\psi(\G_{\tt sc})\geq \frac{2 a b^2 \delta^2  }{1-\theta} =  \frac{2 a^2 b^2 \delta^2  }{4 \alpha}=  \frac{ B^2 \delta^2  }{2 d \alpha},
\end{equation}
where we first plug in $a(1-\theta)=4\alpha$ and then substitute for $a$ and $b$.

\paragraph{Completing the proof of Theorem \ref{t:sconcom_P2} (Communication constraints).}
By proceeding as in Section \ref{s:Proof_conv_P2}, from Lemma~\ref{l:ImpL} and using the inequality~\eqref{eq:inequality:psiGsc} above, we have
\begin{equation}
	\label{eq:lb:assouadtype:sc}
	\sup_{\X \in \mathbb{X}_2(D)} \ep^*(\X , \oO_{{\tt sc}}, T,
        \W_{{\tt com}, r })  
	\geq  \frac{B^2\delta^2}{12\alpha}  \left[1 -
          \sqrt{\frac {\newer{2}}{d}\sum_{i=1}^d\mutualinfo{V(i)}{Y^T}}
\right]. 
\end{equation} 
It remains to bound $\sum_{i=1}^d\mutualinfo{V(i)}{Y^T}$
 to complete the proof. Note that unlike the proof in
 Section~\ref{s:Proof_conv_P2}, the gradient estimates have different
 distributions for different $x$. However, for a point $x$ we can
 still express the gradient estimate $\hat{z}(x) $ of $g_v(x)$ given by $O_v$ as follows:
 abbreviating $f^{\prime+}_i(x):=\frac{\partial f^+_i(x)}{\partial x_i}$ and
 $f^{\prime -}_i(x):=\frac{\partial f^-_i(x)}{\partial x_i}$, we have
\begin{align}\label{e:preprocessing}
 \hat{z}(x)(i)= a  Z_i f^{\prime+}_i(x) + a (1-Z_i) f^{\prime-}_i(x),
\end{align}
where $Z_i\sim\operatorname{Ber}(1/2+\delta v(i))$ and the $Z_i$'s are mutually independent. 
Thus, for a fixed $x$, $\hat{z}(x)$ can be viewed as a function of $\{
Z_i\}_{i \in [d]}$. Furthermore, for a channel $W \in \W_{{\tt com},
  r}$ consider the channel $W^{\prime}_x$ which first passes the
Bernoulli vector $\{Z_i\}_{i \in [d]}$  through the function
$\hat{z}(x)(i)$ and the resulting output is passed through the channel
$W$. This composed channel $W_x$ belongs to $\W_{{\tt com}, r}$, too.

Therefore, we can treat the independent copies of $Z\sim \p_v$
revealed by the oracle as i.i.d. random variables $X_1, ..., X_n$
in Section~\ref{ssec:avg_info}.
Further, note that at time $t$, the query is for a point $x_t$ which is a random
function of $Y^{t-1}$, and so,
$Y^T$ can be viewed as the channel
outputs with adaptively selected channels from $\W_{{\tt com},r}$. Thus, we can apply
the bounds in Theorem~\ref{cor:simple-numbits}. 

Doing so, analogously to the computations in Section~\ref{s:Proof_conv_P2},\footnote{As we have, in both cases, unknown Bernoulli product distribution over $\{-1,1\}^d$ with bias vector $\frac{1}{2}+\delta v$.} we get
\[
\sum_{i \in [d]}I(v(i) \wedge \{Y_i\}_{i \in [T]} )
\sum_{i=1}^d\mutualinfo{V(i)}{Y^T}
\leq c \delta^2 r T,
\] 
for an appropriate constant $c$, which in view
of~\eqref{eq:lb:assouadtype:sc} leads to
\[
	\sup_{\X \in \mathbb{X}_2(D)} \ep^*(\X , \oO_{{\tt sc}}, T,
        \W_{{\tt com}, r }) \geq { \frac{ B^2\delta^2}{12\alpha}
          \cdot} \left[1 - \frac{1}{\sqrt{d}}\cdot\sqrt{\newer{2}cT\delta^2
            r}\right] = \frac{1}{\newer{192}c}\cdot \frac{ B^2}{\alpha T}\cdot
        \frac{d}{r}
\]
the last equality by setting $\delta \eqdef \sqrt{\frac{d}{\newer{8}c
    Tr}}$. Finally, observe that this choice of $\delta$ indeed
satisfies $\delta < \frac{1}{2}\cdot\frac{1-\theta}{1+\theta}$, as
long as $T \geq \newer{2}c\cdot\frac{ B^2}{D^2}\cdot\frac{d }{\alpha^2 r}$. This
completes the proof. \qed

\paragraph{Completing the proof of Theorem \ref{t:sconpriv_P2} (Privacy constraints).}
Proceeding as in the proof of Theorem~\ref{t:sconcom_P2} above, we
have the analogue of~\eqref{eq:lb:assouadtype:sc},
\[
	\sup_{\X \in \mathbb{X}_2(D)} \ep^*(\X , \oO_{{\tt sc}}, T,
        \W_{{\tt priv}, \priv })
	\geq
	 \frac{B^2\delta^2}{12\alpha}  \left[1 -
          \sqrt{\frac {\newer{2}}{d}\sum_{i=1}^d\mutualinfo{V(i)}{Y^T}}\right].
\]
As stated in the proof of Theorem \ref{t:sconcom_P2}, the privatization of the gradient $\hat{z}(x)$ can be viewed as first preprocessing $\{Z_i\}_{i \in [d]}$ and the passing the preprocessed output through the LDP channel. Such a composed channel also belongs to $\W_{{\tt priv}, p}$.
Thus, we can apply the bound in Theorem~\ref{cor:ldp} and proceed as in the proof of Theorem~\ref{t:conpriv_P2} to obtain
\[
	\sum_{i=1}^d\mutualinfo{V(i)}{Y^T} \leq c T \delta^2 \priv^2
\]
where $c>0$ is an absolute constant. Choosing $\delta \eqdef \sqrt{\frac{d}{\newer{8}c T \priv^2}}$, which makes $2\delta$ less than $\frac{1-\theta}{1+\theta}$ for   $T \geq \newer{2}c\cdot \frac{B^2}{D^2}\cdot\frac{d }{\alpha^2 \priv^2}$, for some universal positive constant $c$, then yields
\[
	\sup_{\X \in \mathbb{X}_2(D)} \ep^*(\X , \oO_{{\tt sc}}, T, \W_{{\tt priv}, \priv }) 
	\geq c_0\cdot\frac{B^2}{\alpha T}\cdot \frac{d}{\priv^2}
\]
for some absolute constant $c_0>0$, 
concluding the proof. \qed

\paragraph{Completing the proof of Theorem \ref{t:sconcompt_P2} (Computational constraints).}
As before, we can get
\[
	\sup_{\X \in \mathbb{X}_2(D)} \ep^*(\X , \oO_{{\tt sc}}, T,\W_{\tt obl})
	\geq
	  \frac{B^2\delta^2}{12\alpha}  \left[1 -
          \sqrt{\frac {\newer{2}}{d}\sum_{i=1}^d\mutualinfo{V(i)}{Y^T}}\right].
\]
Recall that we can express the subgradient estimate as in \eqref{e:preprocessing}. Note that for an oblivious sampling channel $W_t$  used at time $t$, specified by a probability vector $(p_j)_{j \in [d]}$, the output
is given by
\eq{\newest{
Y_i = (a Z_{J_t} f^{\prime+}_{J_t}(x) + a(1-Z_{J_t})f^{\prime-}_{J_t}(x))e_{J_t}, }
}
where $J_t=j$ with probability $p_j$.
To proceed, we observe that the Markov relation
\newest{$V \text{---} \{Z_{J_t}, J_t\}_{t\in [T]} \text{---} {Y^T}$} holds.
Indeed, we can confirm this by noting that
\newest{$\{Z_{J_t}\}_{t \in [T]}$} are generated i.i.d. from $\p_V$
and, for each $t\in[T]$, $Y_t$ is a function of \newest{$(Y^{t-1}, Z_{J_t}, J_t)$} and a local randomness $U$
available only to the optimization algorithm which
is independent jointly of $V$ and \newest{$\{Z_{J_t}, J_t\}_{t \in [T]}$.}
It follows that $Y^{T}$ itself is a function of $U$ and \newest{$\{Z_{J_t}, J_t\}_{t \in [T]}$},
which gives 
\begin{align}\label{eq:sc_nice_markovchain}
\newest{\condmutualinfo{V}{Y^T}{\{Z_{J_t}, J_t\}_{t\in [T]}}
\leq \condmutualinfo{V}{U}{\{Z_{J_t}, J_t\}_{t\in [T]}}
=0.}
\end{align}

From the previous observation, we also get
that
the Markov relation
\newest{$V(i) \text{---} \{Z_{J_t}, J_t\}_{t\in [T]} \text{---} {Y^T}$} holds
for every $i\in[d]$.
Thus,
by the data processing inequality for mutual information, we have
\[\newest{
\sum_{i=1}^d\mutualinfo{V(i)}{Y^T}
\leq
\sum_{i=1}^d\mutualinfo{V(i)}{\{Z_{J_t}, J_t\}_{t \in [T]}}.}
\]
Now since vector $(Z_j)_{j \in [d]}$ is a Bernoulli vector, the mutual information on the right-side can be bounded by the same computation as in the proof of Theorem \ref{t:concompt_P2} using Theorem \ref{cor:obl}. \newest{ This follows by observing that for all $t \in [T]$, $(Z_{J_t}, J_t)$ is a function of $Z_{J_t} e_{J_t}$, which  in turn can be seen as a output of the oblivious sampling channel for an input vector $(Z_j)_{j \in [d]}$.} Therefore, we have
\[
\sum_{i=1}^d\mutualinfo{V(i)}{Y^T}
\leq cT\delta^2
\]
for an appropriate constant $c$ and $\delta \leq \frac{1}{6},$
which in view of~\eqref{eq:lb:assouadtype:sc} leads to
\[
	\sup_{\X \in \mathbb{X}_2(D)} \ep^*(\X , \oO_{{\tt sc}}, T, \W_{{\tt com}, r }) 
	\geq  \frac{ B^2\delta^2}{12\alpha}  \left[1 -  \frac{1}{\sqrt{d}}\cdot\sqrt{\newer{2}cT\delta^2 }\right]
	= \frac{1}{c_0}\cdot \frac{d B^2}{\alpha T},
\]
where the last identity is obtained by setting $\delta \eqdef c_1\sqrt{\frac{d}{ T}}$, where $c_0$ and $c_1$ are universal positive constants.
Finally, observe that this choice of $\delta$ indeed satisfies $\delta < \frac{1}{2}\cdot\frac{1-\theta}{1+\theta}$, as long as $T \geq c_2\cdot \frac{B^2}{D^2}\cdot\frac{d^{2}}{\alpha }$, for some universal positive constant $c_2$. This completes the proof. \qed

\section{Adaptivity helps}
  \label{sec:proofs:adapt:advantage}
  In the previous sections we showed for information-constrained first-order optimization over the standard function and oracle classes, adaptive channel selection strategies offer no better minmax convergence guarantees than nonadaptive channel selection strategies.
\newerest{In all the cases, we made this claim in the minmax sense. Namely, we showed that for the worst-case function-oracle pair, adaptive schemes need not help. However, it does not
imply that \newer{adaptivity does not help for \emph{any} function-oracle pair}. In fact, we now exhibit an interesting convex function class and associated oracle for which adaptivity can help. 
}

\newerest{Our example considers the oblivious sampling family $\W_{\tt obl}$.}
Recall that in Randomized Coordinate Descent (RCD), the oracle returns the gradient along a single, randomly chosen coordinate~\citep{nesterov2012efficiency, richtarik2014iteration}. One can consider an adaptive version of this algorithm which allows to choose which coordinate to query the gradient for: we refer to this variant as \emph{Adaptive Coordinate Descent} (ACD).
We provide an example of a function class for which ACD has a strictly better performance than RCD, thereby showing that adaptive channel selection can help. 

\subsection{Mean estimation as an optimization problem.}
The problem we consider entails a structured
$\ell_2$ minimization.
We first define $s$-block sparsity, which is needed to define our function class.
\begin{defn}
	\label{d:sp}
A vector $v \in \R^d$ is \emph{$s$-block sparse} if (i) there exists an $i$ such that $v_j=0$ for all $j\notin\lbrace is+1, \ldots, \min\{ i(s+1), d\} \rbrace$ and (ii) the  nonzero coordinates have the same absolute value in $[0, 1]$. Let  $\B_s$ be the set of all $s$-block sparse vectors in $d$ dimensions.\footnote{For simplicity, we assume throughout that $d/s$ is an integer.}
\end{defn}
\newerest{For $v\in \B_s$ and $\X=[-1,1]^d$ let $f_v\colon\X\to \R$ be the function
$
	f_v(x) =  \norm{x- v}_2^2$, $x\in\X.
$}
Further, we associate with each function $f_v$ an oracle $O_{v}$ as follows. Let $X$ be a random variable over $\{-1, 1\}^d$ with $\E{X}=v$ (i.e., its mean is the $s$-block sparse vector $v$ parameterizing $f_v$). \newerest{Moreover, we assume that each coordinate of $X$ is independent.}
The gradient estimate output of the oracle $O_{v}$ at $x$ and at time $t$ is $2(x -X_{t})$, where  $\{X_{ t}\}_{t=1}^{\infty}$ are i.i.d. random variables with the same distribution as $X$. Note that the expected value of this gradient estimate is $\nabla f(x)$.  
Let $\oO_{{\tt blsp}, s}$ denote the collection of pairs of functions and oracles described above. 

\newerest{As remarked earlier, we have fixed the class of oracles for our example. In our general formulation in Section~\ref{ssec:set-up}, we did not 
even require the oracle to return independent outputs for different queries. The specific oracle above returns independent outputs for every query, and identically distributed outputs for the same query. Furthermore, the outputs are independent across the coordinates and each coordinate takes values $-1$ or $+1$. Interestingly, similar oracles were used in our lower bounds earlier.}

Observe that the first-order optimization described above is the standard $\ell_2$ mean estimation problem cast as an optimization problem, since the function $f_v$ is minimized at $x^\ast \eqdef \E{X}=v$. Moreover, the essential information supplied by the oracle are the i.i.d. samples $X_{t}$ (since the algorithm already knows the queries $x$).

\newest{We will consider the block-sparse function and oracle class $\oO_{{\tt blsp}, s}$  using
the oblivious sampling channel
family $\W_{\tt obl}$  and show that 
adaptive channel selection strategies strictly
outperform the nonadaptive ones.
Towards that, we first derive a lower bound for nonadaptive
strategies,
and then present an adaptive scheme which improves over this
bound.} 

Recall that
  $\displaystyle{\ep^{\rm{}NA\ast}(\X , \oO_{{\tt blsp}, s}, T, \W_{\tt
  obl}) \geq \ep^*(\X , \oO_{{\tt blsp}, s}, T, \W_{\tt obl})}$.
\newest{We will show a \emph{strict} separation between the two quantities: for $s\eqdef \sqrt{d}$, the error incurred by any nonadaptive strategy is at least $\Omega(d^{3/2}/T)$, while there exists an adaptive strategy achieving error $O((d\log d)/T)$.}

\subsection{Lower bound for nonadaptive channel selection strategies}
\label{ssec:nonadaptive_lower}
We show an $\Omega(ds/T)$ lower bound on the error for nonadaptive strategies.
 \begin{thm}\label{t:sparse:nonadapt}
 Let $\X=[-1, 1]^d$. Then,  there exists absolute constants $c_0, c_1,c_2>0$, such that for any $s \geq c_0$ and $T \geq c_1 d,$ we have
\[
 \ep^{\rm{}NA*}(\X , \oO_{{\tt blsp}, s}, T, \W_{\tt obl}) 
\geq 
c_2\cdot \frac{s d}{T}\,.
\]
\end{thm}
\begin{proof}
Let $\delta\in(0,1/2]$ be a parameter to be determined in the course of the proof. 
Let $\V_s \subset\{-1,1\}^s$ be a maximal $(s/4)$-packing in Hamming distance, i.e., a collection of vectors such that $d_H(v, v^\prime) > s/4$ for any two distinct $v, v^\prime\in \cZ$. By the Gilbert--Varshamov bound, we have $|\V_s|\geq 2^{c s}$ for some constant $c\in(0,1)$.
Now define the set $\V \subset \{-1, 1\}^d$ of $d$-dimensional $s$-block sparse vectors as follows:
\[
	\V = \bigcup_{i \in \{1, \ldots d/s\}}\{v^T = (v_1^T, \ldots , v_{d/s}^T):  v_i \in \V_s, v_j= 0  \in \R^s \ \forall j \neq i \}.
\]
That is, $\V$ is the set of all $s$-block sparse vectors such that the non-sparse block contains all possible vectors from $\V_s$. From the definition, we immediately have $|\mathcal{V}|\geq \frac{d}{s}2^{c s}$.

We will restrict ourselves to the subclass of functions $\G_{s, \delta}\subseteq \G_{{\tt blsp}, s}$ consisting of all the functions of the form
 \[
f_{2\delta v}(x)= \norm{ x- 2\delta v}_2^2,\quad v \in \V. 
\]
Fix $v\in\V$. Clearly, the minimizer $x^*$ of $f_{2\delta v}$ is $2\delta v$, for which $f_{2\delta v}(x^*)=0$, and therefore 
\[
	f_{2\delta v}(x)-f_{2\delta v}(x^*)=   \left\| x - 2\delta v\right\|_2^2.
\]
Also, recall from the previous section that the oracle $O_{2\delta v}$ associated with $f_{2\delta v}$ will, upon query $x\in\R^d$, output the gradient estimate $2(x -X_{v})$, where $X_v\in\{-1,1\}^d$ is a random variable with mean $2\delta v$, whose distribution we get to specify. We will choose it as a product distribution over $\{-1,1\}^d$, such that, for every $i\in[d]$,
\begin{align}\label{eq:lb_pmf}
\bPr{ X_v(i)=1 } =   \frac{1+2\delta v(i)}{2}, \quad \bPr{X_v(i)=-1}= \frac{1-2\delta v(i)}{2}.
\end{align}
We can verify that
that $\E{X_{v}}= 2\delta v$.

We will use Fano's method %
to prove the lower bound. Fix any optimization algorithm $\pi$, and denote by $Y^T$ and $\hat{x}\in\R^d$ the corresponding transcript over the $T$ time steps and its eventual output, respectively. Let $V$ be distributed uniformly over $\mathcal{V}$.  
First, we relate the optimization error to the mutual information between $V$ and the messages $Y^T$:
\begin{clm}
	\label{clm:adaptivity:lb:fano}
For $V$ and $Y^T$ as above, we have
\begin{equation}\label{eq:Fano}
	\E{f_{2\delta V}(x)-f_{2\delta V}(x^*)} \geq \frac{s \delta^2}{4}\mleft( 1- \frac{I(V \land  Y^T)+1}{cs +\log (d/s).}\mright).
\end{equation}
\end{clm}
\begin{proof}
By Markov's inequality, we have
\[
	\E{f_{2\delta V}(x)-f_{2\delta V}(x^*)}=\E{ \left\| x - 2\delta V \right\|_2^2} \geq  \frac{s \delta^2}{4} \bPr{ \left\| x - 2V\delta \right\|_2^2 \geq \frac{\sqrt{s} \delta}{2} },
\]
where the expectation is over the uniform choice of $V$
 and the randomness in choosing $x$.
 
Consider the multiple hypothesis testing problem of determining $V$ by observing $Y^T$. For this problem consider the estimator which, after running $\pi$ to obtain an approximate minimizer $\hat{x}$ of $f_{2\delta V}$,  outputs the $V$ which is closest to the estimated $\hat{x}$, denoted by $V(\hat{x})$:
\[
	V(\hat{x}) \eqdef \arg\!\min_{u\in\V} \norm{\hat{x}-2\delta u}_2\,.
\]
We will prove the following bound for the probability of error for this algorithm:
\[
\bPr{V(\hat{x}) \neq V}
\leq \bPr{ \left\| \hat{x} - 2V\delta \right\|_2^2 \geq \sqrt{s} \delta/2 }.
\]
To see this, recall that every distinct $u,u^\prime\in\V$ satisfy $d_H(u, u^\prime) > s/4$, which implies $\norm{2\delta u- 2\delta u'}_2 > \sqrt{s} \delta $. Therefore, whenever $\norm{\hat{x}-2V\delta}_2 <  \sqrt{s} \delta/2$,
the triangle inequality guarantees that, for every $u\in\V$ such that $u\neq V$,
\[
	\norm{\hat{x}-2\delta u}_2 \geq \norm{2\delta u- 2\delta V}_2 - \norm{\hat{x}-2\delta V}_2 > \sqrt{s} \delta/2 > \norm{\hat{x}-2\delta V}_2.
\]
It follows that 
\[
	\bPr{ V(\hat{x}) \neq V } \leq \bPr{ \left\| \hat{x} - 2V\delta \right\|_2^2 \geq \sqrt{s} \delta/2 },
\]
as claimed. 
By Fano's inequality, we also have a lower bound on this error:
\[
	\bPr{ V(\hat{x}) \neq V }  \geq 1- \frac{I(V \land Y^T)+1}{\log |\V|}.
\]
Putting the two together yields~\eqref{eq:Fano}.
\end{proof}
It remains to bound $I(V \land  Y^T)$, which we do next.

\begin{clm}
	\label{clm:adaptivity:lb:info_2}
For $V$ and $Y^T$ as above, we have
$
	I(V \land  Y^T) \leq \frac{4\delta^2 s T}{d}.
$
\end{clm}
\begin{proof}
\newest{Since $\pi$ is a nonadaptive protocol, the random variables $Y^T$ are independent (albeit not necessarily identically distributed). Therefore, by similar arguments as in proving \eqref{eq:sc_nice_markovchain} and denoting \newer{as before} by $e_1,\dots,e_d$ the standard basis vectors, we have
\[
 I(V \land  Y^T) \leq \sum_{t=1}^T I(V \land   X_V(J_t) e_{J_t}),
\]
where $J_t=i$ with probability $p_i,$ for all $i \in [d],$ and $J_{t_1}$ is independent of  $J_{t_2}$.

We will derive a uniform bound for $I(V \land   X_V(J_t)e_{J_t})$ for all $t\in[T]$. To do so, fix any $t\in[T]$, and denote by $W\in\W_{\tt obl}$ the channel used as the $t$th time step and by $(p_i)_{i\in[d]}$ its corresponding distribution over coordinates. 
Denoting by $P_{X_{v^{\prime}}}$ the product distribution described in~\eqref{eq:lb_pmf} (when the underlying vector is $v'$) and recalling the definition of a channel in $\W_{\tt obl}$, we can rewrite $P_{ X_v(J_t ) e_{J_t}\mid v^{\prime}}$, the conditional pmf of $X_v(J_t ) e_{J_t}$ given $V=v^{\prime}$, as follows:
\begin{align*}
	P_{  X_v(J_t ) e_{J_t} \mid v^{\prime}}(e_i) ) 
	&= p_i \cdot P_{X_{v^{\prime}}(i)}(1)
	= p_i \cdot \frac{1+2\delta v'(i)}{2},\\
	P_{  X_V(J_t) e_{J_t}\mid v^{\prime}}(-e_i ) 
	&= p_i \cdot P_{X_{v^{\prime}}(i)}(-1)
	= p_i \cdot \frac{1-2\delta v'(i)}{2},
\end{align*}
for all $i\in[d]$. (In particular, $P_{ X_v(J_t ) e_{J_t}}$ is supported on $2d$ elements.)

Then, by joint-convexity of $D(P \| Q)$, we have
\begin{align*}
I(V \land   X_v(J_t ) e_{J_t}) &= \sum_{v^{\prime} \in \V} P_{V}(v^{\prime})D( P_{ X_v(J_t ) e_{J_t} \mid v^{\prime}} \| P_{ X_v(J_t ) e_{J_t} } ) \\
&\leq \sum_{v^{\prime} \in \V} P_{V}(v^{\prime}) \sum_{i \in [d]} p_i D( P_{X_{v^{\prime}}(i)}  \| \sum_{v^{\prime} \in \V} P_{V}(v^{\prime})P_{X_{v^{\prime}}(i)} )\\
&= \sum_{i \in [d]} p_i\sum_{v^{\prime} \in \V} P_{V}(v^{\prime})  D( P_{X_{v^{\prime}}(i)}  \| \sum_{v^{\prime} \in \V} P_{V}(v^{\prime})P_{X_{v^{\prime}}(i)} ).
\end{align*}
Fixing $i\in[d]$, we now use the fact that 
\[\sum_{v^{\prime} \in \V}P_{V}(v^{\prime})  D( P_{X_{v^{\prime}}(i)}  \|\sum_{v^{\prime} \in \V} P_{V}(v^{\prime}) P_{X_{v^{\prime}}(i)} ) \leq \sum_{v^{\prime} \in \V}P_{V}(v^{\prime})  D( P_{X_{v^{\prime}}(i)}  \| Q ),\]
for every  $Q$ with support $\{-1, 1\}$.
Choosing $Q$ as the uniform distribution over $\{-1, 1\}$, it then suffices to bound $ \sum_{v^{\prime} \in \V} P_{V}(v^{\prime})D( P_{X_{v^{\prime}}(i)} \| Q).$

 Note that $D( P_{X_{v^{\prime}}(i)  } \| Q)  = 0$ unless $i$ belongs to the block of $s$ non-zero coordinates of $v'$. When $i$ belongs to that block, however, we get by upper bounding KL divergence by chi-square divergence that
\eq{
D( P_{X_{v^{\prime}}(i) } \| Q ) 
&\leq \sum_{x \in \{-1, 1\}} \frac{\left( P_{X_{v^{\prime}}(i)  } (x) - Q(x) \right)^2}{ Q(x)}
= \sum_{x \in \{-1, 1\}} \frac{\left( \frac{1+2\delta v'(i) x}{2} - \frac{1}{2} \right)^2}{ 1/2}
=4\delta^2.
}
 Since $V$ is drawn uniformly at random, the probability (over $V$) that the block to which $i$ belongs is the non-sparse one is $s/d$. Consequently,
  $\sum_{v^{\prime} \in \V} P_{V}(v^{\prime})D( P_{X_{v^{\prime}}(i) } \| Q)\leq \frac{4\delta^2s}{d}$. As this holds for every $i\in[d]$, plugging this in our bound for $I(V \land  Y_t)$ leads to
\begin{align*}
I(V \land  Y_t) &\leq \sum_{i \in [d]} p_i\sum_{v^{\prime} \in \V} P_{V}(v^{\prime})  D( P_{X_{v^{\prime}}(i)}  \| Q )
\leq \sum_{i \in [d]} p_i \cdot \frac{4\delta^2s}{d} = \frac{4\delta^2s}{d}.
\end{align*}
Summing over all $t\in[T]$ then proves the claim.}
\end{proof}
In order to conclude the proof, we combine Claims~\ref{clm:adaptivity:lb:fano} and~\ref{clm:adaptivity:lb:info_2}, to obtain
\begin{align*}
 \E{f_{V}(x)-f_{V}(x^*)}  
	&\geq {s\delta^2} \left(1-\frac{4T\delta^2s/d+1}{cs +\log (d/s)}\right)\\
	& \geq {s\delta^2} \left(1-\frac{4T\delta^2s/d+1}{cs}\right)\\
	& \geq {s\delta^2} \left(1-\frac{8T\delta^2}{c d}\right)
\end{align*}
where the final inequality holds for $\delta^2 \geq \frac{d}{4T s}$.
We now choose $\delta^2 =\frac{cd}{16T}$, \newest{which is a valid choice for $s 
\geq \frac{4}{c},$} to get
\[
\E{f_{V}(x)-f_{V}(x^*)} \geq \frac{c}{32}\cdot \frac{sd }{T},
\]
where we require $T \geq \frac{cd}{4}$ to ensure $\delta^2 \leq \frac{1}{4}$, which, in turn, is essential for \eqref{eq:lb_pmf} to define a valid pmf.
\end{proof}
\subsection{Adaptivity helps}
\newest{We now prove a $O((d\log(d/s)+s^2)/T)$ upper bound on the error for adaptive strategies,
by exhibiting a specific adaptive channel selection strategy
and optimization procedure 
we term \emph{Adaptive Coordinate Descent} (ACD),
denoted $\pi_{\tt ACD}$.}

First, note that the only new information that the
oracles present at each iteration is about the random variable $X$ with $\E{X}=v$
underlying the oracle associated with some function $f_v(x)=\|x-v\|_2^2$ in our family $\oO_{{\tt blsp}, s}$.
Thus, the problem at hand becomes that of estimating the mean $v$
using independent copies of $X$. See Algorithm~\ref{a:ACD} for a detailed description.

Keeping this in mind, our adaptive channel selection strategy is divided in two phases, each
making  $T/2$ queries to the oracle:\footnote{We assume for simplicity
that $T/2$, $Ts/(2d),$ and $T/(2s)$ are integers.} the {\em exploration phase} and
the {\em exploitation phase}.
In the exploration phase, we select each
block's first coordinate as a representative coordinate for that
block and query each representative coordinate $Ts/(2d)$
times. At the end of this phase, an estimate of the mean is formed for
each representative coordinate.
Next, we select the block whose representative
coordinate has the
sample mean
\new{with the highest absolute value}. Then, in
the exploitation phase each coordinate of the selected  block is
queried $T/(2s)$ times.

\newest{Our optimization algorithm
estimates the means of coordinates} in the selected block
using the sample mean of the values received in the 
 exploitation phase.
For the rest of the coordinates, the mean estimate is zero.
Finally, our algorithm returns the overall estimated mean vector
as the estimated minimizer of the function. 

\new{Recall that in RCD, the oracle returns the gradient along a
randomly chosen coordinate. In contrast, ACD gets gradient for a particular
coordinate in each round, and the choice of the coordinates used in
the exploitation phase depends on the observations of the exploration
phase.
\newer{Also, we note that it is possible
to interpret our procedure as a coordinate descent algorithm.} However, for the ease of presentation, we simply retain the form above.

\begin{algorithm}[ht!]
\SetAlgoLined
\DontPrintSemicolon
 \tcc{Exploration phase: the first $T/2$ oracle queries}
 \For{$i=1$ \KwTo $d/s$}
 {
   $T_1 \gets 1 + (i-1) \cdot \frac{Ts}{2d}$, $T_2\gets i \cdot \frac{Ts}{2d}$\;
   \For{$t=T_1$ \KwTo $ T_2$}
   {
      Sample $X_t((i-1) \cdot s +1)$, the $((i-1) \cdot s +1)$th coordinate of the gradient estimate at time $t$\;
      Query the oracle for arbitrary $x \in [-1, 1]^d$\;
      Sample the $((i-1) \cdot s +1)$th coordinate of the gradient estimate at  time $t$\;
      Recover $X_t((i-1) \cdot s +1)$ from the $((i-1) \cdot s +1)$th coordinate of the gradient estimate\;
    }
    Compute 
        \[	
         		\hat{X}(i) \gets \sum_{t=T_1}^{T_2}  X_t((i-1) \cdot s +1)
	      \]
	    ;
 }
 Set 
   	\[
		i^\ast \gets \arg\!\max_{i \in [d/s]} |\hat{X}(i)|
	\]
	and $\mathcal{I}\gets \{i^\ast, \ldots, i^\ast +(s-1)\}$.\;
	
	\tcc{Exploitation phase: the last $T/2$ oracle queries}
    \For{$i \in \mathcal{I}$}
    {
       \new{ Set $T_1 \gets T/2+ 1+ (i-1)\cdot \frac{T}{2s}$ and $T_2 \gets T/2+ i \cdot \frac{T}{2s}$}\;
        \For{$t=T_1$ \KwTo $ T_2$}
        {
            Query the oracle for arbitrary $x \in [-1, 1]^d$\;
            Sample the $i$th coordinate of the gradient estimate at time $t$\;
            Recover $X_t(i)$ from the $i$th coordinate of the gradient estimate\;
        }
        Compute
	   \new{ \[
		    \hat{Y}(i) \gets \frac{2s}{T} \sum_{t=T_1}^{T_2}X_t(i)
	    \]}
    }
    \For{$i \in [d]\setminus \mathcal{I}$}
    {
      $\hat{Y}(i) \gets 0$
    }
    \KwResult{$\hat{Y}= [\hat{Y}(1), \ldots, \hat{Y}(d) ]^T$}
 \caption{Adaptive Coordinate Descent $\pi_{\tt ACD}$}\label{a:ACD}
\end{algorithm}
}

\new{The performance of $\pi_{\tt ACD}$ is characterized by the result below.}

\begin{thm}\label{t:sparse:adapt}
 Fix any $1\leq s\leq d$, and $(f,O)\in\oO_{{\tt blsp}, s}$.\footnote{
\new{That is, $f(x)=f_v(x)=\|x-v\|^2$ for some $v$ with block sparsity
 structure and $O$ gives independent copies of random variable $X$
 with $\E{X}=v$}.} 
 Let $\hat{Y}\in\R^d$ be the point returned by Algorithm~\ref{a:ACD} after $T$ oracle queries to $O$. Then, 
 \[
 	\E{f(\hat{Y})} \leq \frac{36 d\ln\frac{d}{s}+2s^2}{T}\,,
 \]
\end{thm}

\begin{proof}
Fix $(f,O)$ as in the statement, so that $f$ is parameterized by some $s$-block sparse vector $v\in[-1,1]^d$, with $f(x) = \norm{x-v}_2^2$; and $O$ corresponds to the distribution of some random variable $X$ over $\{-1,1\}^d$ with mean $\E{X}=v$. 
For simplicity, and without loss of generality, we assume that the block of non-sparse coordinates for the mean vector is $\{1, \ldots,  s\}$. Further, let\footnote{Recall from Definition \ref{d:sp} that all the non-zero mean coordinates have the same mean value in absolute value. Therefore, $|\E{X(i)}|=|\delta|$ for all $i \in \{1, \ldots, s \}$.}
$\delta \colon = \E{X(1)}$. Using the same notation as in the description of Algorithm \ref{a:ACD}, denote by $\mathcal{I}$ the index of the coordinates in the block selected by the algorithm. Moreover, let $\mathcal{J} \eqdef [d]\setminus \mathcal{I}$ be the set of remaining coordinates. We can rewrite the error as
\begin{align}\label{eq:first_split}
\E{f(\hat{Y})} 
&=\E{\norm{\hat{Y} - v}_2^2}
=\E{ \sum_{i \in \mathcal{I}} (\hat{Y}(i) -v(i))^2} + \E{ \sum_{i \in \mathcal{J}}(\hat{Y}(i) -v(i))^2} .
\end{align}
We will bound both terms separately. To handle the first, recall that, for all $i \in  \mathcal{I}$, we have
$
	\hat{Y}(i)= \frac{2s}{T}\sum_{t=T_1}^{T_2}X_t(i),
$
where  $T_1 = \frac{T}{2}+ 1+ (i-1)\cdot \frac{T}{2s}$ and $T_2= \frac{T}{2}+ i \cdot \frac{T}{2s}.$
Therefore, for all $i \in [d],$\new{
\begin{align*}
\E{ (\hat{Y}(i) -v(i))^2 \mathbf{1}_{\mathcal{I}}(i) \mid \mathcal{I} }
&= \frac{2s}{T} \E{ (X_{T_1}(i) - v(i))^2 }\mathbf{1}_{\mathcal{I}}(i)
\leq \frac{2s}{T}  \E{X_{T_1}(i) ^2 }\mathbf{1}_{\mathcal{I}}(i)
\leq \frac{2s}{T}\mathbf{1}_{\mathcal{I}}(i),
\end{align*}}
where the first equality follows from the fact that the sequence of random vectors $\{X_t\}_{t=T/2+1}^{T}$ is i.i.d.\ and independent of the random set $ \mathcal{I}$, along with the fact that $\E{X_{T_1}(i)  }=v(i)$; and the second inequality is because $X_{T_1}(i) \in [-1, 1].$ Since $|\mathcal{I}|= s$, by the law of total expectation we get
\new{
\begin{align}\label{eq:First_Term}
\E{ \sum_{i \in \mathcal{I}} (\hat{Y}(i) -v(i))^2}
= \sum_{i=1}^d \E{ (\hat{Y}(i) -v(i))^2 \mathbf{1}_{\mathcal{I}}(i) } \leq \frac{2s^2}{T}.
\end{align}}
We claim that the second term of the RHS can be bounded as follows:
\begin{equation}\label{eq:Second_Term}
	\E{ \sum_{i \in \mathcal{J}}(\hat{Y}(i) -v(i))^2} \leq \frac{36d}{T}\ln\frac{d}{s}
\end{equation}
\newcommand{\THR}{\new{R}}
To see why, set $\THR \eqdef \frac{36d}{s}\ln\frac{d}{s}$, so that our goal is to show that $\E{ \sum_{i \in \mathcal{J}}(\hat{Y}(i) -v(i))^2} \leq s\THR/T$. 
First, for all $i \in \mathcal{J}$,
$\hat{Y}(i)=0$, and so we have
\[
	\E{ \sum_{i \in \mathcal{J}}(\hat{Y}(i) -v(i))^2} = \E{ \sum_{i \in \mathcal{J}}v(i)^2}
	= s\delta^2\bPr{ \mathcal{I}\neq \{1,\dots,s\} }.
\]
The last equality follows from the fact that if $\mathcal{I}$ is the correct block (which we assumed was $\{1,\dots,s\}$), then $\mathcal{J}$ only contains coordinates $i$ for which the mean $v(i)=v_i=0$; while if $\mathcal{I}$ is not the correct block, then all $s$ coordinates of that block are in $\mathcal{J}$, and each of them has $|v(i)|=|\delta|$.

If $|\delta|\leq \sqrt{\THR/T}$, we are done, as then $s\delta^2\bPr{ \mathcal{I}\neq \{1,\dots,s\} } \leq s\THR/T$, which is what we wanted. 
Thus, we hereafter assume $|\delta|> \sqrt{\THR/T}$ and want to bound $\bPr{ \mathcal{I}\neq \{1,\dots,s\} }$, which by the description of our algorithm is exactly the probability that $i^\ast \neq 1$. That is,
\[
	\bPr{ |\hat{X}(1)| \leq \max_{2\leq i\leq d/s} |\hat{X}(i)| },
\]
where $\hat{X}(1),\hat{X}(2),\dots,\hat{X}(d/s)$ are independent random variables, with $\hat{X}(2),\dots,\hat{X}(d/s)$ being identically distributed as the sum of \new{$N\eqdef \frac{Ts}{2d}$ independent 1-subgaussian r.v.'s} and $\hat{X}(1)$ being the sum of $N$ i.i.d.\ random variables in $[-1,1]$ with mean $\delta$.
On the one hand, by a standard argument (see for instance \cite{boucheron2013concentration}), one can check that
\[
	\E{\max_{2\leq i\leq d/s} |\hat{X}(i)|} \leq \sqrt{2 N \ln\frac{d}{s} } < \frac{1}{3}N\sqrt{\frac{\THR}{T}}
\]
where 
the last inequality used our setting of $\THR \geq \frac{36d}{s}\ln \frac{d}{s}$. On the other hand, 
\[
  \left|\E{\hat{X}(1)}\right| = N|\delta| > N\sqrt{\frac{\THR}{T}} 
\]
and therefore we have
\begin{align*}
	\lefteqn{\bPr{ |\hat{X}(1)| \leq \max_{2\leq i\leq d/s} |\hat{X}(i)| }}
\\
&\leq \bPr{ \left\{ |\hat{X}(1)| \leq  \frac{2}{3} N |\delta|  \right\} \cup \left\{ \max_{2\leq i\leq d/s} |\hat{X}(i)| \geq  \frac{2}{3} N |\delta| \right\} } \\	
	&\leq \bPr{ |\hat{X}(1)| \leq  \frac{2}{3} N |\delta|  } + \bPr{ \max_{2\leq i\leq d/s} |\hat{X}(i)| \geq  \frac{2}{3} N |\delta| } \\	
	&\leq \bPr{ |\hat{X}(1)| \leq \frac{2}{3}\mleft|\E{\hat{X}(1)}\mright| } + \bPr{ \max_{2\leq i\leq d/s} |\hat{X}(i)| > \sqrt{2 N \ln(d/s) } + \frac{1}{3}N |\delta|  }
\end{align*}
We handle both terms separately. By symmetry, we can assume without loss of generality that $\E{\hat{X}(1)}\geq 0$, and so
\[
  \bPr{ |\hat{X}(1)| \leq \frac{2}{3}\mleft|\E{\hat{X}(1)}\mright| } 
  \leq \bPr{ \hat{X}(1) \leq \frac{2}{3}\E{\hat{X}(1)} }
  \leq e^{-\frac{\E{\hat{X}(1)}^2}{18N} } = e^{-\delta^2 N/18 }
\]
by a Hoeffding bound. Note that we then have
\begin{align*}
    s\delta^2 \bPr{ |\hat{X}(1)| \leq \frac{2}{3}\mleft|\E{\hat{X}(1)}\mright| }
    &\leq s\delta^2 e^{-\delta^2 N/18 }
 = \frac{s\THR}{T}\cdot \frac{36 d}{s \THR} \cdot \frac{\delta^2 N}{18} e^{-\delta^2 N/18 }    
\\
&\leq \frac{s\THR}{T} \cdot e^{-1}
\end{align*}
\new{since $N=Ts/2d$, $\THR\geq 36d/s$.} 

Turning to the second term, by a standard concentration bound for the maximum of subgaussian r.v.'s and using the fact that each $\hat{X}(i)$, for $i\geq 2$, is $N$-subgaussian, we get
\[
    \bPr{ \max_{2\leq i\leq d/s} |\hat{X}(i)| > \sqrt{2 N \ln(d/s) } + \frac{1}{3}N |\delta|  }
    \leq e^{-\frac{(\delta N/3)^2}{2N} } = e^{-\delta^2 N/18 }
\]
and we conclude as before that
\[
    s\delta^2 \bPr{ \max_{2\leq i\leq d/s} |\hat{X}(i)| > \sqrt{2 N \ln(d/s) } + \frac{1}{3}N |\delta|  }
    \leq \frac{s\THR}{eT}\,.
\]

\noindent This shows that, in this case, 
\begin{equation}\label{eq:second_term_2}
	s\delta^2\bPr{ \mathcal{I}\neq \{1,\dots,s\} } \leq 2e^{-1}\cdot s\THR/T \leq s\THR/T
\end{equation} as well. 
Plugging \eqref{eq:First_Term} and \eqref{eq:second_term_2} in \eqref{eq:first_split},
\new{we get 
\new{$
\E{\norm{\hat{Y}-\mu}_2^2} \leq \frac{36 d\ln(d/s)+2s^2}{T}
$}, proving the theorem}.
\end{proof}
Combining Theorems~\ref{t:sparse:nonadapt} and~\ref{t:sparse:adapt}, for $\X=[-1, 1]^d$ and $s=\sqrt{d}$ we obtain a strict separation between nonadaptive and adaptive strategies: 
\[
    \ep^{\rm{}NA*}(\X , \oO_{{\tt blsp}, s}, T, \W_{\tt obl}) 
\gtrsim \frac{d^{3/2}}{T}, 
\qquad\text{ but }\qquad 
\ep^{\ast}(\X , \oO_{{\tt blsp}, s}, T, \W_{\tt obl}) \leq \frac{20 d\ln d}{T}
\]
for $T=\Omega(d)$. \newerest{Note that the separation between adaptive and nonadaptive schemes hold for all \newer{$\log d \ll s \ll d$}, but the multiplicative gain in convergence rate is maximized for $s\approx \sqrt{d}$.}

\bibliography{IEEEabrv,tit2018,references}

\begin{thebibliography}{ACGMMTZ16}

\bibitem[ABRW12]{agarwal2012information}
Alekh Agarwal, Peter~L Bartlett, Pradeep Ravikumar, and Martin~J Wainwright.
\newblock {Information-Theoretic Lower Bounds on the Oracle Complexity of
  Stochastic Convex Optimization}.
\newblock {\em IEEE Transactions on Information Theory}, 5(58):3235--3249,
  2012.

\bibitem[ACGMMTZ16]{abadi2016deep}
Martin Abadi, Andy Chu, Ian Goodfellow, H~Brendan McMahan, Ilya Mironov, Kunal
  Talwar, and Li~Zhang.
\newblock Deep learning with differential privacy.
\newblock pages 308--318. ACM, 2016.

\bibitem[ACT20a]{acharya2020general}
Jayadev Acharya, Cl{\'e}ment~L Canonne, and Himanshu Tyagi.
\newblock General lower bounds for interactive high-dimensional estimation
  under information constraints.
\newblock {\em arXiv preprint arXiv:2010.06562v3}, 2020.

\bibitem[ACT20b]{ACT20:IT2}
Jayadev {Acharya}, Cl\'ement~L. {Canonne}, and Himanshu {Tyagi}.
\newblock {Inference Under Information Constraints II: Communication
  Constraints and Shared Randomness}.
\newblock {\em IEEE Transactions on Information Theory}, 66(12):7856--7877,
  2020.

\bibitem[ADSFS19]{acharya2019distributed}
Jayadev Acharya, Christopher De~Sa, Dylan~J Foster, and Karthik Sridharan.
\newblock {Distributed Learning with Sublinear Communication}.
\newblock {\em arXiv:1902.11259}, 2019.

\bibitem[AGLTV17]{alistarh2017qsgd}
Dan Alistarh, Demjan Grubic, Jerry Li, Ryota Tomioka, and Milan Vojnovic.
\newblock {{QSGD}: Communication-efficient SGD via gradient quantization and
  encoding}.
\newblock {\em Advances in Neural Information Processing Systems}, pages
  1709--1720, 2017.

\bibitem[ASYKM18]{agarwal2018cpsgd}
Naman Agarwal, Ananda~Theertha Suresh, Felix Xinnan~X Yu, Sanjiv Kumar, and
  Brendan McMahan.
\newblock {{cpSGD}: Communication-efficient and differentially-private
  distributed {SGD}}.
\newblock {\em Advances in Neural Information Processing Systems}, pages
  7564--7575, 2018.

\bibitem[BGMNW16]{BGMNW:16}
Mark Braverman, Ankit Garg, Tengyu Ma, Huy~L. Nguyen, and David~P. Woodruff.
\newblock {Communication lower bounds for statistical estimation problems via a
  distributed data processing inequality}.
\newblock {\em Proceedings of ACM Symposium on the Theory of Computing (STOC'
  16)}, pages 1011--1020, 2016.

\bibitem[BLM13]{boucheron2013concentration}
St{\'e}phane Boucheron, G{\'a}bor Lugosi, and Pascal Massart.
\newblock {\em {Concentration inequalities: A nonasymptotic theory of
  independence}}.
\newblock \hspace{-0.2cm}Oxford university press, 2013.

\bibitem[Bub15]{bubeck2015convex}
S{\'e}bastien Bubeck.
\newblock {Convex optimization: Algorithms and complexity}.
\newblock {\em Foundations and Trends{\textregistered} in Machine Learning},
  8(3-4):231--357, 2015.

\bibitem[CK{\"O}20]{chen2020breaking}
Wei-Ning Chen, Peter Kairouz, and Ayfer {\"O}zg{\"u}r.
\newblock Breaking the communication-privacy-accuracy trilemma.
\newblock {\em arXiv preprint arXiv:2007.11707}, 2020.

\bibitem[DJW14]{duchi2014privacy}
John~C Duchi, Michael~I Jordan, and Martin~J Wainwright.
\newblock Privacy aware learning.
\newblock {\em Journal of the ACM (JACM)}, 61(6):1--57, 2014.

\bibitem[DR19]{DuchiR19}
John~C. Duchi and Ryan Rogers.
\newblock Lower bounds for locally private estimation via communication
  complexity.
\newblock In Alina Beygelzimer and Daniel Hsu, editors, {\em Conference on
  Learning Theory, {COLT} 2019, 25-28 June 2019, Phoenix, AZ, {USA}}, volume~99
  of {\em Proceedings of Machine Learning Research}, pages 1161--1191. {PMLR},
  2019.

\bibitem[FTMARRK20]{faghri2020adaptive}
Fartash Faghri, Iman Tabrizian, Ilia Markov, Dan Alistarh, Daniel Roy, and Ali
  Ramezani-Kebrya.
\newblock Adaptive gradient quantization for data-parallel sgd.
\newblock {\em Advances in Neural Information Processing Systems}, 2020.

\bibitem[GDDKS20]{girgis2020shuffled}
Antonious~M Girgis, Deepesh Data, Suhas Diggavi, Peter Kairouz, and
  Ananda~Theertha Suresh.
\newblock Shuffled model of federated learning: Privacy, communication and
  accuracy trade-offs.
\newblock {\em arXiv preprint arXiv:2008.07180}, 2020.

\bibitem[GKMM19]{gandikota2019vqsgd}
Venkata Gandikota, Daniel Kane, Raj~Kumar Maity, and Arya Mazumdar.
\newblock vqsgd: Vector quantized stochastic gradient descent.
\newblock {\em arXiv preprint arXiv:1911.07971}, 2019.

\bibitem[HHWY19]{huang2019optimal}
Zengfeng Huang, Ziyue Huang, Yilei Wang, and Ke~Yi.
\newblock Optimal sparsity-sensitive bounds for distributed mean estimation.
\newblock {\em Advances in Neural Information Processing Systems}, pages
  6371--6381, 2019.

\bibitem[KR18]{konevcny2018randomized}
Jakub Kone{\v{c}}n{\'y} and Peter Richt{\'a}rik.
\newblock Randomized distributed mean estimation: Accuracy vs. communication.
\newblock {\em Frontiers in Applied Mathematics and Statistics}, 4:62, 2018.

\bibitem[LKH20]{lin2020achieving}
Chung-Yi Lin, Victoria Kostina, and Babak Hassibi.
\newblock Achieving the fundamental convergence-communication tradeoff with
  differentially quantized gradient descent.
\newblock {\em arXiv preprint arXiv:2002.02508}, 2020.

\bibitem[MT20a]{mayekar2020limits}
Prathamesh Mayekar and Himanshu Tyagi.
\newblock Limits on gradient compression for stochastic optimization.
\newblock {\em Proceedings of the {IEEE} International Symposium of Information
  Theory (ISIT' 20)}, 2020.

\bibitem[MT20b]{mayekar2020ratq}
Prathamesh Mayekar and Himanshu Tyagi.
\newblock {RATQ}: A universal fixed-length quantizer for stochastic
  optimization.
\newblock {\em Proceedings of the International Conference on Artificial
  Intelligence and Statistics (AISTATS' 20)}, pages 1399--1409, 2020.

\bibitem[Nem95]{nemirovski1995information}
Arkadi Nemirovsky.
\newblock {Information-based complexity of convex programming}.
\newblock 1995.
\newblock {Available Online }
  \url{http://www2.isye.gatech.edu/ne-mirovs/Lec_EMCO.pdf}.

\bibitem[Nes12]{nesterov2012efficiency}
Yu~Nesterov.
\newblock Efficiency of coordinate descent methods on huge-scale optimization
  problems.
\newblock {\em SIAM Journal on Optimization}, 22(2):341--362, 2012.

\bibitem[Nes13]{nesterov2013introductory}
Yurii Nesterov.
\newblock {Introductory lectures on convex optimization: A basic course}.
\newblock {\em Springer Science and Business Media}, 87, 2013.

\bibitem[NY83]{nemirovsky1983problem}
Arkadi Nemirovsky and David~Borisovich Yudin.
\newblock {Problem complexity and method efficiency in optimization.}
\newblock {\em Wiley series in Discrete Mathematics and Optimization}, 1983.

\bibitem[RKFR19]{ramezani2019nuqsgd}
Ali Ramezani-Kebrya, Fartash Faghri, and Daniel~M Roy.
\newblock Nuqsgd: Improved communication efficiency for data-parallel sgd via
  nonuniform quantization.
\newblock {\em arXiv preprint arXiv:1908.06077}, 2019.

\bibitem[RT12]{richtarik2012parallel}
Peter Richt{\'a}rik and M~Tak{\'a}c.
\newblock Parallel coordinate descent methods for big data optimization, arxiv
  e-prints.
\newblock {\em arXiv preprint arXiv:1212.0873}, 2012.

\bibitem[RT14]{richtarik2014iteration}
Peter Richt{\'a}rik and Martin Tak{\'a}{\v{c}}.
\newblock Iteration complexity of randomized block-coordinate descent methods
  for minimizing a composite function.
\newblock {\em Mathematical Programming}, 144(1):1--38, 2014.

\bibitem[SFDLY14]{seide20141}
Frank Seide, Hao Fu, Jasha Droppo, Gang Li, and Dong Yu.
\newblock 1-bit stochastic gradient descent and its application to
  data-parallel distributed training of speech dnns.
\newblock {\em Fifteenth Annual Conference of the International Speech
  Communication Association}, 2014.

\bibitem[SSR20]{safaryan2020uncertainty}
Mher Safaryan, Egor Shulgin, and Peter Richt{\'a}rik.
\newblock Uncertainty principle for communication compression in distributed
  and federated learning and the search for an optimal compressor.
\newblock {\em arXiv preprint arXiv:2002.08958}, 2020.

\bibitem[SVK20]{subramani2020enabling}
Pranav Subramani, Nicholas Vadivelu, and Gautam Kamath.
\newblock Enabling fast differentially private {SGD} via just-in-time
  compilation and vectorization.
\newblock {\em arXiv preprint arXiv:2010.09063}, 2020.

\bibitem[SYKM17]{suresh2017distributed}
Ananda~Theertha Suresh, Felix~X Yu, Sanjiv Kumar, and H~Brendan McMahan.
\newblock {Distributed mean estimation with limited communication}.
\newblock {\em Proceedings of the International Conference on Machine Learning
  (ICML' 17)}, 70:3329--3337, 2017.

\bibitem[Yu97]{yu1997assouad}
Bin Yu.
\newblock {Assouad, Fano, and Le Cam}.
\newblock In {\em Festschrift for Lucien Le Cam}, pages 423--435. Springer,
  1997.

\end{thebibliography}
\bibliographystyle{alpha-all}
\appendix
  \section{Proof of Lemma \ref{lemma:B_alpha}}
From the strong convexity  of $f$, we have
\[
f\mleft(\frac{x+y}{2}\mright) - \frac{\alpha}{2}\mleft\lVert\frac{x+y}{2}\mright\rVert_2^2 \leq \frac{1}{2}f(x) -\frac{\alpha}{4}\norm{x}_2^2 + \frac{1}{2}f(y) - \frac{\alpha}{4} \norm{y}_2^2,
\quad \forall\, x,y\in\X, 
\]
which upon reorganizing and using the fact that $2\norm{x}_2^2+2\norm{x}_2^2-\norm{x+y}_2^2 = \norm{x-y}_2^2$
can be seen to be
equivalent to
\[
	\frac{\alpha}{4}\norm{x-y}_2^2 \leq f(x)+f(y)  -2f\mleft(\frac{x+y}{2}\mright).
\]
Further, by Lipschitz continuity of $f$ in the \new{$\ell_2$ norm, we have
\[
	f(x)+f(y)  -2f\mleft(\frac{x+y}{2}\mright) \leq B\norm{x-y}_2.
        \]
        }
Upon combining the previous two bounds, we obtain
\[
\frac{\alpha}{4} \leq \frac{f(x)+f(y)  -2f\mleft(\frac{x+y}{2}\mright)}{ \norm{x-y}_2^2} \leq
\new{
\frac{B}{\norm{x-y}_2}},
\] 
which completes the proof upon substituting  $y=0$ and $x$ such that $|x(i)|=D$ for all $i\in[d]$
giving $\norm{x-y}_2=d^{1/2}D$.

\section{Upper Bounds for $p=1$ under communication constraints.}\label{app:scheme}
We now provide a scheme which matches the lower bound of Theorem \ref{t:concom_P2} for $\ell_1$ norm and $r$-bits communication constraints, for optimization for the family of convex functions. 
Our scheme divides the entire horizon of $T$ iterations into $Tr/d$ different phases. For any phase $t \in [Tr/d]$, the same point $x_t$ in the domain is queried $d/r$ times. For each of the $d/r$ queries in a phase, we use $r$-bit quantizers to quantize  different coordinates of the subgradient output. At a high level, we want to use these $r$ bits to send $1$ bit each for $r$ different coordinates, sending $1$ bit for each coordinate across the phases. However,
there is one technical difficulty. We have not assumed that making queries for the same point
gives identically distributed random variables. We circumvent this difficulty using random
permutations to create unbiased estimates for the subgradients. 

Specifically, for a permutation $\sigma\colon [d] \to [d]$ chosen uniformly at random using public randomness, we select the coordinates $\sigma(1+(i-1)\cdot r)$ to $\sigma(i\cdot r)$  of the subgradient estimate $\hat{g}_i$ supplied by the oracle for the $i$th query in the $t$th phase (i.e., $i$th time we query the point $x_t$) and quantize all of these coordinates using an $1$-bit unbiased quantizer for the interval $[-B, B]$. Note that such a quantizer can be formed
since $\norm{\hat{g}_i}_\infty\leq B$. 

Using this procedure, the quantized gradient for every query in each phase can be stored in $r$ bits. Furthermore, using all the $d/r$ quantized estimates received in a phase, we can create an estimate of the subgradient by simply adding all the estimates. Denote by $\bar{Q}_t$ our subgradient estimate in the $t$th phase. Then, 
\[
\bar{Q}_t=\sum_{i=1}^{d} Q_{\pi(i)}(\hat{g}_i)e_{\pi(i)},
\]
where
$\hat{g}_i$ is the subgradient estimate returned by the oracle when we query $x_t$ for the $i$th time and $Q_i$ is a $1$-bit unbiased estimator of the $i$th coordinate of gradient estimate given below: For all vectors $g$, such that $\norm{g}_{\infty}\leq B$, we have
\[
Q_i(g)
 = 
\begin{cases} 
B &\text{ w.p. }  \quad\frac{g(i)+B}{2B}\\
-B   &\text{ w.p. }  \quad\frac{B-g(i)}{2B}
\end{cases}. 
\] 
Then, we use $\bar{Q}_t$ to update $x_t$ to $x_{t+1}$ using stochastic mirror descent with mirror map 
\[
\phi_a(x) \colon= \frac{\norm{x}_a^2}{a-1},
\]
where
$a =\frac{2\log d}{2\log d-1}$. Recall that for a mirror map $\Phi$, the Bregman divergence associated with $\Phi$ is defined as \[D_{\Phi}(x,y) \colon= \Phi(x)-\Phi(y)- \langle \nabla \Phi(y), x-y \rangle.
\]

\begin{algorithm}[H]
\SetAlgoLined
\DontPrintSemicolon
 \For{$t=1$ \KwTo $Tr/d$}
 {
  
   \For{$i=1$ \KwTo $ d/r$}
   {
 {\color{red1} At Center:}\;
      Query the oracle for  $x_t$\;
      
      \vspace{0.2cm}
       {\color{blue} At Oracle:}\;
        Output the $r$-bit vector of $1$-bit unbiased estimates of the $r$ coordinates $\{1+(i-1)\cdot r, \ldots ,i\cdot r \}$ of $\hat{g}_i(x_t)$ given by 
        \[
          \bar{Q}_t \gets \sum_{j=1+(i-1)\cdot r}^{i \cdot r} Q_{\pi(j)}(\hat{g}_i(x_t))e_{\pi(j)}
        \]
        \;
         {\color{red1} At Center:}\;
  $x_{t+1} \gets \arg\min_{x \in \X} (\eta_t \langle x, \bar{Q}_t \rangle) +D_{\Phi_a}(x, x_t))$ \;
    }
 \KwResult{ $\frac{\sum_{i=1}^{T}x_t}{T}$}
 }
 
 \caption{$\pi^*$ Optimal Scheme for Communication constrained optimization for $\ell_1$ convex family}
\end{algorithm}

\begin{thm}
 For $ r  \in \N$, we have
\[
 \sup_{\X \in \mathbb{X}_1(D)} \ep^*(\X , \oO_{{\tt c}, 1}, T, \W_{{\tt com}, r  }) 
\leq 
\frac{c_0DB \sqrt{\log d}}{\sqrt{T}} \cdot \sqrt{\frac{d}{d \wedge r}} 
\]
for every $D>0$.
\end{thm}
\begin{proof}
  Note that our first order optimization algorithm $\pi^*$ uses $Tr/d$ iterations. Moreover, the subgradient estimates $\bar{Q}_t$ are unbiased and have their infinity norm bounded by $B$.
  Namely, we have obtained an unbiased subgradient oracle which produces estimates
  with infinity norm bounded by $B$. Thus, using the standard analysis of
  mirror descent using noisy subgradient oracle for optimization over an $\ell_1$ ball
  with mirror map
  $\phi_a(x) \colon= \frac{\norm{x}_a^2}{a-1}$  (see Remark \ref{r:c_p1}), the proof is complete.
\end{proof}

\end{document}